\documentclass[11pt]{amsart}
\usepackage[letterpaper,margin=1in]{geometry}
\usepackage{url}
\usepackage{amssymb, amsmath, amsthm, amssymb, amsfonts}
\usepackage{tikz}
\theoremstyle{plain}
\newtheorem{theorem}{Theorem}[section]
\newtheorem{lemma}[theorem]{Lemma}
\newtheorem{proposition}[theorem]{Proposition}
\newtheorem{assumption}[theorem]{Assumption}
\theoremstyle{definition}
\newtheorem{definition}[theorem]{Definition}
\theoremstyle{remark}
\newtheorem{remark}[theorem]{Remark}
\usetikzlibrary{decorations.pathmorphing}
\usetikzlibrary{shapes,shadows,fit} 
\usetikzlibrary{calc}
\usepackage[utf8]{inputenc}
\usepackage{mathtools}
\newcommand{\coloneq}{\coloneqq}
\usepackage{algorithm}
\usepackage{algpseudocode}
\usepackage{graphicx} 
\usepackage{caption}  
\usepackage{subcaption} 
\usepackage{makecell} 
\usetikzlibrary{positioning}
\definecolor{edgecol}{RGB}{38,70,120}
\definecolor{pos}{RGB}{176,24,43}
\definecolor{neg}{RGB}{38,70,120}
\usepackage{hyperref}
\usepackage{cleveref}
\usepackage{diagbox}
\usepackage{siunitx}   
\usepackage{caption}
\usepackage{multirow, booktabs}
\newtheorem{rulethm}{Rule}
\usetikzlibrary{arrows.meta, positioning}
\usetikzlibrary{decorations.pathmorphing, decorations.text}
\usetikzlibrary{decorations.pathreplacing, positioning}
\DeclareMathOperator{\F}{\mathcal{F}}
  
\begin{document}

\title[\texttt{E-IRMGL+$\Psi$}]{Graph Laplacian assisted regularization method under noise level free heuristic and statistical stopping rule}
\author[H. Bajpai]{Harshit Bajpai}
\author[A. K. Giri]{Ankik Kumar Giri}
\thanks{
\hspace{-0.64cm} Department of Mathematics, Indian Institute of Technology Roorkee, Roorkee, Uttarakhand 247667, India.\\[3pt]
E-mail: \texttt{harshit\_b@ma.iitr.ac.in} (Harshit Bajpai), \texttt{ankik.giri@ma.iitr.ac.in} (Ankik Kumar Giri).\\[3pt]
This work was supported by the Anusandhan National Research Foundation (ANRF), India, under Grant No.\ CRG/2022/00548.
}
\maketitle
\begin{quote}
{{\em \bf Abstract.}}
In this work, we address the solution of both linear and nonlinear ill-posed inverse problems by developing a novel graph-based regularization framework, where the regularization term is formulated through an iteratively updated graph Laplacian. The proposed approach operates without prior knowledge of the noise level and employs two distinct stopping criteria namely, the heuristic rule and the statistical discrepancy principle. To facilitate the latter, we utilize averaged measurements derived from multiple repeated observations.
We provide a detailed convergence analysis of the method in statistical prospective, establishing its stability and regularization properties under both stopping strategies. The algorithm begins with the computation of an initial reconstruction using any suitable techniques like Tikhonov regularization (Tik), filtered back projection (FBP) or total variation (TV), which is used as the foundation for generating the initial graph Laplacian. The reconstruction is made better step by step using an iterative process, during which the graph Laplacian is dynamically re-calibrated to reflect how the solution's structure is changing.
Finally, we present numerical experiments on X-ray Computed Tomography (CT) and phase retrieval CT, demonstrating the effectiveness and robustness of the proposed method and comparing its reconstruction performance under both stopping rules.
\end{quote}
\vspace{.3cm}

\noindent

\begin{description}
\item[\textbf{Keywords}] data-assisted regularization, unknown noise level, statistical inverse problem, graph Laplacian operator, ill-posed problems, phase retrieval, medical imaging. \vspace{.1cm}
\item[\textbf{MSC codes}] 47J06, 05C90, 92C50, 47A52.
\end{description}

\section{Introduction}
\noindent Inverse problems are encountered in various scientific and engineering disciplines, with the objective being to determine an unknown value from indirect and noisy data. Examples include image reconstruction in medical tomography~\cite{barrett1992image}, parameter identification in PDEs~\cite{jadamba2011inverse}, and signal recovery in geophysics~\cite{kearey2013introduction}. Mathematically, these problems are often cast as linear or nonlinear operator equations of the type
\begin{equation}\label{Model eqn}
    \F(u) = v,
\end{equation}  
where $\F : \mathcal{D}(\F) \subset \mathcal{U} \simeq \mathbb{R}^n \to \mathcal{V} \simeq \mathbb{R}^m$ represents the 
forward operator mapping the solution space $\mathcal{U}$ to the data space $\mathcal{V}$. 
Here, $\mathcal{U}$ and $\mathcal{V}$ possess the canonical 
inner product $\langle \cdot, \cdot \rangle$ and induced norm $\|\cdot\|$. In practical scenarios, the exact data $v$ is unavailable and only noisy measurement $v^\delta$ can be observed. Due to their ill-posed nature, inverse problems (\ref{Model eqn}) are highly sensitive to data perturbations, making the stable recovery of the true solution from noisy observations a central computational challenge.
When the noise level $\delta = \|v - v^\delta\|$
is known, various variational~\cite{scherzer2009variational} and iterative regularization methods~\cite{kaltenbacher2008iterative} have been created. Among them, Landweber-type schemes~\cite{fu2023two, hanke1995convergence, jin2025convergence, mittal2025convergence, neubauer2000landweber} have attracted extensive study. Recently, data-driven regularization~\cite{aspri2020data, aspri2020data1, bajpai2025stochastic, zhou2025convergence} has gained prominence for incorporating \emph{a priori} information. In a similar spirit, graph-based regularization methods~\cite{bajpai2025convergence, bianchi2024improved, bianchi2025data} have emerged for linear ill-posed problems, where a graph constructed from the signal encodes essential features of the ground truth via the associated graph Laplacian which is  employed as a guiding operator within the regularization framework.

As demonstrated in~\cite{lou2010image}, imaging tasks like deblurring and tomography perform poorly when the graph $\mathcal G$ is constructed directly from noisy data $v^\delta$, since \( v^\delta \) resides in a different space from the ground-truth image \( u^\dagger \). In order to solve this problem, the authors suggested a preprocessing step that maps the data through a reconstruction operator \( \Psi : \mathcal{V} \to \mathcal{U} \), and the graph is then constructed from \( \Psi(v^\delta) \). The operator \( \Psi \) transfers information from the measurement space to the solution space and can be implemented, for instance, using Tikhonov filtering~\cite{engl1996regularization} or filtered back-projection (FBP)~\cite{kak2001principles}.
Building on~\cite{lou2010image}, the work in~\cite{bianchi2025data} proposed a graph-based variational regularization method, denoted as \(\texttt{graphLa+}\Psi\), which was later extended in~\cite{bajpai2025convergence} to an iterative method equipped with an \emph{a posteriori} stopping rule, thereby alleviating the computational challenges of variational formulations. In these frameworks, preprocessing is performed via a family of reconstruction operators  
\(
\Psi_\Theta : \mathcal{V} \to \mathcal{U},
\)
parameterized by \(\Theta = \Theta(\delta, v^\delta)\), where the parameters may depend on both the noise level and the observed data. The pair \((\Psi_\Theta, \Theta)\) is defined in a general and flexible setting, without the requirement that it constitutes a convergent regularization method.

While iterative methods are effective for addressing ill-posed inverse problems, their inherent semi-convergence behavior necessitates an appropriate stopping criterion to prevent overfitting to noise. A conventional choice is the \emph{discrepancy principle}~\cite{morozov1966solution}, which guarantees regularization when the noise level~$\delta$ is known. However, in many practical scenarios, accurately estimating~$\delta$ is infeasible. To address this limitation, \emph{heuristic stopping rules}~\cite{bajpai2024hanke,jin2017heuristic,kindermann2020convergence,real2024hanke} have been developed, which determine the stopping index directly from the noisy data~$v^\delta$ without requiring explicit noise information. Empirical studies indicate that these rules can perform on par with, or even surpass, methods relying on explicit noise knowledge, though they often require restrictive assumptions on the noise structure. Alternatively, averaging multiple unbiased measurements has been explored to approximate stable solutions, as systematically analyzed in~\cite{harrach2020beyond,harrach2023regularizing,jahn2021modified, jin2023dual} for linear regularization methods with unknown noise in Hilbert spaces.

Graph-based regularization methods have demonstrated promising potential in imaging and related applications, yet the field remains in its early stages, offering promising direction for advancement. Motivated by this, the present work extends the graph-based iterative regularization framework (\texttt{IRMGL+$\Psi$})~\cite{bajpai2025convergence} along four key directions. First, we generalize the approach beyond the linear setting to handle nonlinear ill-posed problems under suitable assumptions, thereby broadening both its theoretical foundation and practical applicability. Second, we introduce a heuristic stopping rule that depends solely on the observed data, eliminating the need for prior noise knowledge. To address its limitations, we further develop a statistical discrepancy principle, derived empirically yet likewise independent of explicit noise information. Both stopping criteria exploit repeated noisy measurements, which form a central element of our design. Finally, to validate the proposed approach, we carry out comprehensive numerical experiments on both X-ray CT and phase retrieval CT problems. These experiments are designed to examine and compare the behavior of the proposed method under two different stopping rules. The results demonstrate superior reconstruction accuracy and stability, underscoring the potential of the proposed framework to advance modern imaging techniques in scientific and medical applications.

To realize this idea, we assume that measurements can be repeatedly acquired under identical experimental conditions, while the true underlying solution remains fixed. This setting aligns with the well-established engineering practice of \emph{signal averaging}, which leverages repeated measurements to mitigate random errors (see \cite{lyons1997understanding} for an introduction and \cite{hassan2010reducing} for a comprehensive survey). Accordingly, we consider a sequence of independent and identically distributed $\mathcal{V}$-valued random variables $v_1, v_2, \ldots,$ within the probability space $(\Omega, \mathbb F, \mathbb P)$, each serving as an unbiased observation of the exact but unknown data $v$. For any $m \geq 1$, we define the mean
\begin{equation}\label{eqn:empricial_mean}
   \hat{v}^{(m)} := \frac{1}{m} \sum_{i=1}^m v_i,   
\end{equation}
which provides an estimator of $v$. Let $\mathbb{E}[\cdot]$ denote the expectation and assuming that $\mathbb{E}[v_1] = v$ with
\[
0 < \sigma^2 := \mathbb{E}\bigl[\|v_1 - v\|^2\bigr] < \infty,
\]
it follows that
\[
\mathbb{E}\bigl[\|\hat{v}^{(m)} - v\|^2\bigr] = \frac{\sigma^2}{m}.
\]
Hence, averaging over $m$ independent measurements reduces the variance $\sigma^2$ by a factor of $m$, yielding a progressively more accurate approximation of the true data as $m \to \infty$.

Building upon the above framework of empirical averaging, we develop a novel and enhanced iterative regularization scheme assisted by the graph Laplacian, referred to as \texttt{E-IRMGL+$\Psi$}. In this approach, the empirical mean $\hat{v}^{(m)}$ serves as the input data for the reconstruction process. As a preliminary step, we introduce a family of reconstruction mappings 
\[
    \Psi_{\theta} : \mathcal{V} \to \mathcal{U},
\]
where the parameter $\theta := \theta(m, \hat{v}^{(m)})$ may depend on both the number of measurements $m$ and the averaged data $\hat{v}^{(m)}$. The proposed iterative scheme is then formulated as
\begin{equation}\label{main iterative scheme}
\begin{cases}
     u_{k+1}^{(m)} = u_k^{(m)} - \alpha_k^{(m)} \F'(u_k^{(m)})^*\big(\F(u_k^{(m)}) - \hat{v}^{(m)}\big) - \beta_k^{(m)} \Delta_{u_k^{(m)}}u_k^{(m)}, \\[1ex]
     u_0^{(m)} = \Psi_\theta(\hat{v}^{(m)}),
\end{cases}
\end{equation}
where $k \geq 1$, $\F'(u_k^{(m)})^*$ denotes the adjoint of the Fr\'echet derivative $\F'(u_k^{(m)})$, $\Delta_{u}$ is the data-dependent graph Laplacian constructed from $u$ (for further information, see Subsection~\ref{subsec: Graph Theory}), and $\alpha_k^{(m)} > 0$ and $\beta_k^{(m)} \geq 0$ denote the step size and regularization weight, respectively. A comprehensive discussion of the method’s formulation, theoretical properties, and computational aspects is provided in Section~\ref{Sec-3} and Subsection~\ref{subsec: The method}, while a schematic illustration of the overall method is shown in Fig. 1.
 It is important to note that the iteration in \eqref{main iterative scheme} can be interpreted as a gradient descent method applied to the functional
\[
\mathcal J(u) := \frac{1}{2} \left (\|\F(u) - \hat v^{(m)}\|^2 + \beta \langle u, \Delta_u u \rangle \right),
\]
where $\beta \geq 0$ is a regularization parameter. The first term enforces data fidelity by penalizing the discrepancy between the model output $\F(u)$ and the measured data $\hat v^{(m)}$. The second term acts as a graph-based regularizer that incorporates structural information encoded in the graph Laplacian $\Delta_u$.
Computing the Fr\'echet derivative of $\mathcal J$ with respect to $u$ yields
\[
\nabla \mathcal J(u) = \F'(u)^*(\F(u) - \hat v^{(m)}) + \beta \, \Delta_u u.
\]
Thus, a gradient descent update with step size $\alpha_k^{(m)} > 0$ naturally leads to
\[
u_{k+1}^{(m)} = u_k^{(m)} - \alpha_k^{(m)} \big(\F'(u_k^{(m)})^*(\F(u_k^{(m)}) - \hat v^{(m)}) + \beta \Delta_{u_k^{(m)}}u_k^{(m)} \big),
\]
which coincides with the scheme defined in \eqref{main iterative scheme}, after allowing the parameters $\alpha_k^{(m)}$ and $\beta_k^{(m)}$ to vary across iterations. Hence, the proposed method can be rigorously viewed as a graph-regularized gradient descent for minimizing $\mathcal J(u)$.

The convergence analysis in this work differs significantly from existing studies~\cite{bajpai2025convergence, bianchi2025data}. First, the use of multiple repeated measurements requires stochastic techniques, in contrast to the deterministic settings of \cite{bajpai2025convergence, bianchi2025data}. Second, the regularization term depends explicitly on both the empirical mean and the number of measurements $m$, adding complexity to the stability analysis and necessitating mild but essential assumptions. Third, to our knowledge, this is the first study to examine heuristic and statistical stopping rules within graph-based regularization. These advancements together constitute the main novel contributions of this work.
\begin{figure}\label{fig:abstract illustration}
    \centering
\begin{tikzpicture}[>=Stealth, node distance=2cm, scale=1]
\draw[thick] (0,0) ellipse (2cm and 3cm);
\draw[thick] (7,0) ellipse (2cm and 3cm);
\draw[fill=gray!30, thick] (-0.3,-2.2) circle (0.6cm);
\draw[fill=gray!30, thick,
      decorate, decoration={random steps, segment length=6pt, amplitude=2pt}]
      (6.7,1.3) circle (1.2cm);
\node at (8,-0.1) {i.i.d $v_i$'s};
\node[circle, fill=black, inner sep=1pt, label=right:$\Bar{u}$] (bar u) at (-1.3,-1) {};
\node[circle, fill=black, inner sep=1.2pt, label= above:$\Psi(\hat v^{(m)})$] (psi hat) at (-0.1,1.8) {};
\node[circle, fill=black, inner sep=1.2pt, label=right:$u_1^{(m)}$] (u1) at (0.3,1.1) {};
\node[circle, fill=black, inner sep=1.2pt, label=right:$u_2^{(m)}$] (u2) at (0.3,0.2) {};
\node[circle, fill=black, inner sep=1.2pt, label=right:$u_{k_m^*}^{(m)}$] (uk heuristic) at (-0.35,-0.8) {};
\node[circle, fill=black, inner sep=1.2pt, label=right:$u_{k_m}^{(m)}$] (uk stat) at (0.85,-0.9) {};
\node[circle, fill=black, inner sep=1.2pt, label=below:$u^\dagger$] (u) at (-0.3,-2.2) {};

\node[circle, fill=black, inner sep=1.5pt, label=right:$v$] (v) at (6,-0.5) {};

\coordinate (vhat) at (6.7,1.2);
\node[circle, fill=yellow, inner sep=2pt, label=above:$\hat v^{(m)}$] at (vhat) {};

\coordinate (v1) at ($(vhat)+(-0.6,0.6)$);
\coordinate (v2) at ($(vhat)+(-0.2,0.8)$);

\node[circle, fill=black, inner sep=1pt, label=above:$v_1$] at (v1) {};
\node[circle, fill=black, inner sep=1pt, label=above:$v_2$] at (v2) {};

\coordinate (v3) at ($(vhat)+(0.2,0.9)$);
\coordinate (v4) at ($(vhat)+(0.9,-0.2)$);

\node[circle, fill=black, inner sep=1pt] at (v3) {};
\node[circle, fill=black, inner sep=1pt] at (v4) {};

\coordinate (v5) at ($(vhat)+(0.4,-0.8)$);
\coordinate (v6) at ($(vhat)+(-0.3,-0.8)$);

\node[circle, fill=black, inner sep=1pt] at (v5) {};
\node[circle, fill=black, inner sep=1pt] at (v6) {};
\coordinate (vm) at ($(vhat)+(-0.9,-0.5)$);
\node[circle, fill=black, inner sep=1pt, label=above:$v_m$] at (vm) {};

\draw[->, thick, bend right=40] (u) to node[midway, above] { $\F(\cdot)$ } (v);

\draw[->] (vhat) to[bend right=15] node[midway, above] { $\Psi$ } (psi hat);

\draw[dotted, bend left=40, ->] (psi hat) to (u1);
\draw[dotted, bend left=40, ->] (u1) to (u2);
\draw[dotted, ->] (u2) to (uk heuristic);
\draw[dotted, ->] (u2) to (uk stat);


\draw[decorate, decoration={coil, aspect=-0.1, amplitude=2mm, segment length=4mm}, ->, blue, thick] 
    (uk heuristic) -- (u);

\draw[decorate, decoration={snake, amplitude=1.3mm, segment length=3mm}, ->, red, thick] 
    (uk stat) -- (u);

\node at ($(uk heuristic)!0.5!(uk stat) + (0.5,-0.8)$) { \tiny $m \to \infty$};
\node at ($(uk heuristic)!0.5!(uk stat) + (-1.15,-0.6)$) { \tiny $m \to \infty$};
\node (mid) at (-1,1.5) {};
\draw[decorate, decoration={zigzag, amplitude=2mm, segment length=8mm}, ->] (psi hat) .. controls(mid) .. (bar u); \draw[decorate, decoration={zigzag, amplitude=2mm, segment length=8mm}, ->] (psi hat) .. controls(mid) .. node[midway, above] {\tiny $ \hspace{-0.7cm} m \to \infty$} (bar u);
\node at (0,-3.5) {$\mathcal{U}$};
\node at (7, -3.5) {$\mathcal{V}$};
\node[align=center, right=2cm of u2] {\texttt{E-IRMGL+\(\Psi\)}};

\end{tikzpicture}
\caption{A schematic representation of the \texttt{E-IRMGL+$\Psi$} approach with two noise level free stopping rules. Independent and identically distributed samples $v_1, v_2, \ldots, v_m$ (black points near $\hat v^{(m)}$) are averaged to obtain $\hat v^{(m)}$ (yellow). The primary reconstructor $\Psi$ is not usually a regularization technique, which is reflected in the piecewise linear trajectory of $\Psi(\hat v^{(m)})$ as $m \to \infty$. Starting from $\Psi(\hat v^{(m)})$, the iterative procedure generates approximations $u_k^{(m)}$. The approximate reconstructions obtained via the heuristic and statistical stopping rules are denoted by $u_{k_m^*}^{(m)}$ and $u_{k_m}^{(m)}$, respectively. As $m \to \infty$, both sequences converge to the ground-truth solution $u^\dagger$. The convergence of $u_{k_m^*}^{(m)}$ is indicated by the blue coiled arrow, whereas the convergence of $u_{k_m}^{(m)}$ is represented by the red wavy arrow.
}
\end{figure}
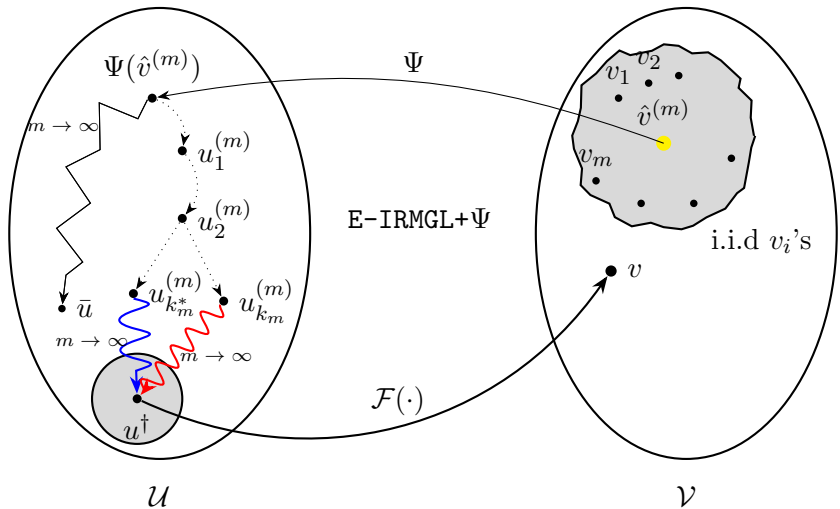

The paper is organized as follows. 
Section~\ref{Sec-3} begins with the essential definitions and preliminaries, followed by the key assumptions of the study. The section then introduces the heuristic and statistical stopping strategies that constitute the main components of the proposed framework. Section~\ref{sec:IRMGP} presents the algorithmic formulation under the statistical stopping rule and establishes stability and convergence results. Section~\ref{sec:heuristic} develops the algorithm under heuristic stopping rule and its convergence analysis. The performance of the proposed approach, along with a comparison of both stopping rules, is demonstrated through numerical experiments in Section~\ref{sec: numerical}. Finally, Section~\ref{sec:conclusions} summarizes the main contributions and outlines potential directions for future research. To improve the readability and overall flow of the paper, several technical proofs supporting the main convergence results are deferred to the Appendix~\ref{sec:appendix}.

\section{The \protect\texttt{E-IRMGL+$\Psi$} method}\label{Sec-3}
In this section, we investigate the iterative scheme introduced in \eqref{main iterative scheme} and provide a comprehensive description. Two stopping rules, formulated independently of any \emph{a priori} knowledge of the noise level, are introduced. We begin by recalling essential notions from graph theory and their connection to image representations (see \cite{keller2021graphs} for a modern introduction).
\subsection{Images and graphs}\label{subsec: Graph Theory}
\begin{definition}\label{def:2.1}
Let $\mathcal S = \{a_1, a_2, \ldots, a_{|\mathcal S|}\}$ be a finite set, $|\mathcal S|$ being its cardinality. A graph over $\mathcal S$ is a pair $\mathcal G=(\mathcal S, e_w)$, where $e_w: \mathcal S \times \mathcal S \to [0,\infty)$ is an edge-weight function satisfying  
\begin{itemize}
    \item Symmetry: $e_w(a,b)=e_w(b,a)$ for all $a,b \in \mathcal S$,  
    \item No self-loops: $e_w(a,a)=0$ for all $a \in \mathcal  S$.  
\end{itemize}
The edge set is $T := \{(a,b)\in \mathcal S\times \mathcal S \mid e_w(a,b)\neq 0\}$. Nodes $a,b\in \mathcal S$ are connected, denoted $a\sim b$, if $e_w(a,b)>0$, where $e_w(a,b)$ quantifies the strength of their connection.
\end{definition}
\begin{definition}\label{def:2.2}
For any function $\mathbf{u}: \mathcal S \to \mathbb{R}$, the graph Laplacian $\Delta \mathbf{u}: \mathcal S \to \mathbb{R}$ is defined by  
\begin{equation}\label{Eqn: Graph Laplacian}
    \Delta \mathbf{u}(a) := \sum_{a \sim b} e_w(a,b)\bigl(\mathbf{u}(a) - \mathbf{u}(b)\bigr).
\end{equation}
In matrix form,  
\begin{equation}\label{Delta = D-W}
    \Delta = D - W,
\end{equation}
where $D$ denotes the diagonal degree matrix with entries 
\[d_{ii} = \sum_{j=1}^{|\mathcal S|} e_w(a_i,a_j)\quad \text{ and } W = [e_w(a_i,a_j)]_{i,j=1}^{|\mathcal S|}\] is the weight matrix. 
\end{definition}
The operator in \eqref{Delta = D-W} coincides with the action in \eqref{Eqn: Graph Laplacian}.
Note that, defining a graph requires specifying a set of nodes $\mathcal S$ together with an edge-weight function $e_w$. We now describe how such a graph can be constructed from an image.

An image can be viewed as a grid of pixels $a \in \mathcal S$, where each pixel is indexed by $a=(i_a,j_a)$ with $i_a \in \{1,\ldots,p\}$ and $j_a \in \{1,\ldots,q\}$, denotes its horizontal and vertical axes, respectively. Thus, the node set can be written as
\[
\mathcal{S}
\coloneq \bigl\{\, a = (i_a, j_a) \; : \; i_a \in \{1,\ldots,p\},\; j_a \in \{1,\ldots,q\} \,\bigr\}.
\]
For the sake of keeping things simple, we take grayscale images, represented as a function
\(
\mathbf{u}: \mathcal S \to [0,1],
\)
assigning an intensity to each pixel, where $0$ and $1$ correspond to black and white, respectively, and intermediate values denote shades of gray. A common approach to define the edge weight between two pixels incorporates both spatial proximity and intensity similarity is
\begin{equation}\label{eqn: edge weight fun}
e_{w_{\mathbf{u}}}(a,b) = g(a,b) \, h_{\mathbf{u}}(a,b),
\end{equation}
where $g(a,b)$ captures the geometric relationship of $\mathcal S$, typically
\[
g(a,b) = \mathbf{1}_{(0,R]}(d(a,b)),
\]
with $R>0$ controlling the neighborhood radius, $\mathbf{1}_{(0,R]}$ the indicator function, and $d(a,b)$ a distance metric (e.g., $d(a,b)=|i_a-i_b|+|j_a-j_b|$ or $d(a,b)=\max\{|i_a-i_b|,|j_a-j_b|\}$). The function $h_{\mathbf{u}}(a,b)$ accounts for intensity similarity, commonly using a Gaussian kernel
\[
h_{\mathbf{u}}(a,b) = \exp\Big(-\frac{|\mathbf{u}(a)-\mathbf{u}(b)|^2}{\lambda}\Big), \quad \lambda>0.
\]
For detailed discussions on constructing graph weights from grayscale images using such kernels, see \cite{bianchi2024improved, bianchi2025data, bronstein2017geometric, calatroni2017graph, gilboa2009nonlocal, lou2010image}. Assume that   \( R < \infty \). Then, each node \( a \in \mathcal S \) is connected to at most \( D_N \) neighboring nodes, i.e.,
\begin{equation*}
    \deg(a) := \#\{\, b \in \mathcal S : g(a,b) \neq 0 \,\} \le D_N, \qquad \text{for all } a \in \mathcal S,
\end{equation*}
where \( \deg(a) \) denotes the degree of node \( a \).
A simple illustrative example of graph construction from an image is presented in Figure~\ref{fig:graph from image}.
\begin{figure}[H]
  \centering
  \includegraphics[width=\textwidth]{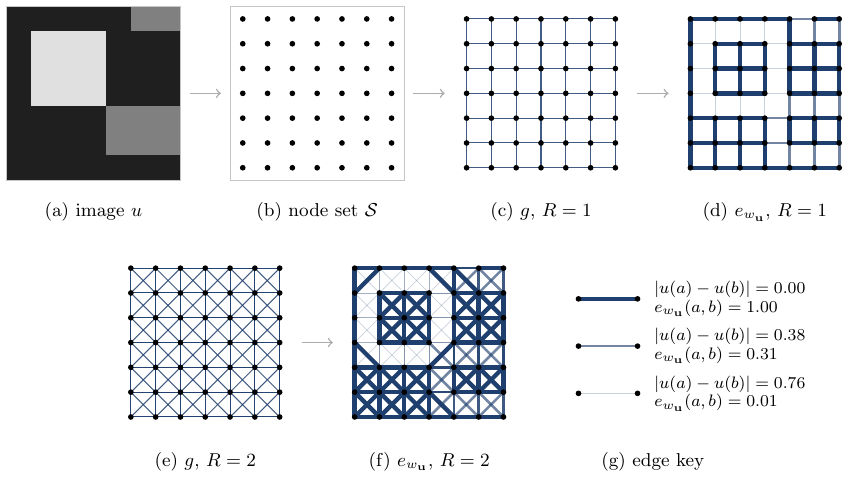}
  \caption[From an image to its graph Laplacian]{
   Panel~(a) depicts a \(7 \times 7\) pixel image, where \(p=q=7\) and \(|\mathcal{S}|=49\). The image \(u\) assumes three distinct values: \(0.12\) for the background, \(0.50\) for the two mid-gray regions, and \(0.88\) for the bright square. Panel~(b) shows the node set, each pixel contributing
    exactly one node, drawn as a black disc.
      As $g$ takes values in $\{0,1\}$ and $h_{\mathbf u}$ in $(0,1]$,
    one has $e_{w_{\mathbf u}}(a,b)\in[0,1]$ for all $a,b\in\mathcal S$. Panel~(c) shows the graph carrying the
    geometric factor alone,
    $g(a,b)=\mathbf 1_{(0,R]}(\|a-b\|_{1})$, for $R=1$: two nodes
    are joined if and only if $0<\|a-b\|_{1}\le1$, that is, if they are
    horizontally or vertically adjacent, and every such edge carries the
    magnitude one. 
    Panel~(d) shows the graph carrying the full weight $e_{w_{\mathbf u}}$ with $R=1$ and
    $\lambda=0.35$. The thickness and the opacity of a segment are increasing
    functions of $e_{w_{\mathbf u}}(a,b)$: an edge is drawn the more heavily, the more
    strongly the two pixels it joins are connected. Since $h_{\mathbf u}$ is
    decreasing with $h_{\mathbf u}(0)=1$ and $h_{\mathbf u}(t)\to 0$ as
    $t\to\infty$, a thick edge means that the two pixels have nearly equal
    intensity, and a faint edge that their intensities differ markedly.
    Panel~(g) calibrates this correspondence on the three contrasts occurring
    in the present image. 
    Panels~(e) and~(f) repeat the two previous
    steps with $R=2$, for which the neighbours of a node are the eight pixels. The graph then has $226$ edges and
    interior nodes have degree $12$. Comparison of (c) with (e) shows that $R$
    governs the sparsity of the graph, and comparison of (d) with (f) that the
    intensity factor continues to sever the connections across the boundaries
    of the image for the larger stencil as well. For the present example $\Delta_u$ is a $49\times49$ matrix
    with $217$ nonzero entries when $R=1$ and $501$ when $R=2$, and
    $\|\Delta_u\|\le 2 D_N$.}
  \label{fig:graph from image}
\end{figure}
Throughout the paper, images are modeled as real-valued functions defined on a common finite set of pixels. More precisely, we fix a finite node set $\mathcal{S}$, arranged on a regular grid and identify the image space with
\(
\mathcal{U}:=\mathbb{R}^{\mathcal{S}}\cong\mathbb{R}^{|\mathcal{S}|},
\)
endowed with the standard Euclidean inner product. We denote the resulting weighted graph and its associated graph Laplacian by
\[
 \mathcal  G_u:=(\mathcal{S},e_{w_u}),\qquad \Delta_u:=D_u-W_u,
\]
respectively, following Definitions~\ref{def:2.1} and \ref{def:2.2}. Hence, the node set $\mathcal{S}$ remains fixed for all images, while the edge weights, and consequently the graph structure, depend on the underlying image $u$.
\subsection{Assumptions and stopping rules}
In this subsection, we first introduce some necessary assumptions which are essential for the subsequent analysis.
\begin{assumption}\label{assump:Fi}
For every \(u, w \in \mathcal{B}_{3\wp}(u_0^{(m)}) \subset \mathcal{D}(\F)\), the following properties hold, 
where \(\mathcal{B}_{\wp}(u_0^{(m)})\) denotes the closed ball of radius~\(\wp\) centered at~\(u_0^{(m)}\).
\begin{enumerate}\label{ass:primary}
    \item[(A1)] The operator $\F$ is Fr\'echet differentiable with derivative $\F'(u)$. Moreover, the mapping $u \mapsto \F'(u)$ is continuous.
    
    \item[(A2)] The derivatives are uniformly bounded, i.e.,
    \(
    \|\F'(u)\| \leq B,
    \)
    for some constant $B > 0$ independent of $u$.
    
    \item[(A3)] (\textbf{Tangential cone condition}) There exists a constant $\eta \in [0,1)$ such that
    \[
    \|\F(u) - \F(w) - \F'(w)(u - w)\| \leq \eta \, \|\F(u) - \F(w)\|.
    \]
    \item[(A4)] There exists a constant $L \geq 0$ such that
    \[
    \|\F'(u) - \F'(w)\| \leq L \, \|u - w\|.
    \]
    \item [(A5)]\label{ass:A5}  There exists $u^\dagger \in \mathcal{B}_{\wp}(u_0^{(m)})$ such that $\F(u^\dagger)=v$.
\end{enumerate}
\end{assumption}

In Assumption~\ref{ass:primary}, conditions (A1) and (A2) ensure that the derivative 
$\mathcal{F}'(u)$ is locally bounded and continuous within the closed ball 
$\mathcal{B}_{3\wp}(u_0^{(m)})$. When $\mathcal{F}$ is a bounded linear operator, 
conditions (A1)--(A4) of Assumption~\ref{ass:primary} are automatically satisfied with 
$\eta = 0$. In the nonlinear setting, condition (A3) imposes a restriction on 
the degree of nonlinearity of $\mathcal{F}$, and this assumption has been verified 
for various classes of nonlinear inverse problems~\cite{hanke1995convergence}. 
Condition (A4) corresponds to the well-known notion of $L$-smoothness, which requires 
that the derivative $\mathcal{F}'(u)$ is uniformly controlled by a parameter, thereby 
preventing arbitrary variations. This property plays a fundamental role in establishing 
convergence guarantees for most gradient-type methods~\cite{bottou2018optimization} 
and has been extensively validated in practical applications~\cite{jin2013iteratively}.
\begin{remark}
  The condition (A5) of the aforementioned assumption  guarantees that the exact solution $u^\dagger$ is in close proximity to the initial reconstruction  $u_0^{(m)}=\Psi_\theta(\hat v^{(m)})$, and is compatible with the forward model. For several standard choices of the reconstruction operator $\Psi_\theta$, this assumption is naturally satisfied. For instance, when $\Psi_\theta$ is chosen to be the Tikhonov regularizer, i.e.,
\[
 \Psi_\theta^{\mathrm{Tik}}(\hat v^{(m)}) = \arg\min_{u \in \mathcal{D}(\F)}\Bigl\{\tfrac{1}{2}\|\F(u)-\hat v^{(m)}\|^2 + \tfrac{\theta}{2}\|u\|^2\Bigr\},
\]
where $\theta:= \theta(m, \hat v ^{(m)})$ is the regularization parameter. It follows from the fact that, under standard assumptions Tikhonov regularization along with the appropriate regularization parameter $\theta$ becomes a convergent regularization method.  Consequently, for sufficiently small values of~$\theta$, the exact solution $u^\dagger$ lies within the ball $\mathcal{B}_{\wp}(u_0^{(m)})$ for a suitably chosen radius $\wp>0$. Further technical details can be found in~\cite{bajpai2025convergence}.
\end{remark}
\noindent Building upon these structural conditions, we next impose assumptions on the measurement data.  
\begin{assumption}\label{assump:iid}
Let $\{v_i\}_{i \geq 1}$ be a sequence of i.i.d.
$\mathcal{V}$-valued random variables defined on a probability space 
$(\Sigma, \mathbb{F}, \mathbb{P})$. We assume that they are unbiased with
\(
\mathbb{E}[v_1] = v,
\)
and possess finite, nonzero variance
\[
0 < \sigma^2 := \mathbb{E}\bigl[\|v_1 - v\|^2\bigr] < \infty.
\]
\end{assumption}
Then, for the empirical mean \eqref{eqn:empricial_mean} we obtain  
\[
\mathbb{E}\bigl[\|\hat{v}^{(m)} - v\|^2\bigr]   
= \frac{1}{m^2} \sum_{i,j=1}^m \mathbb{E}\bigl[\langle v_i - v,\, v_j - v\rangle\bigr]   
= \frac{1}{m^2} \, \sum_{i=1}^{m}\mathbb{E}\bigl[\|v_i - v\|^2\bigr]   
= \frac{\sigma^2}{m}.
\]  
This calculation shows that the variance is reduced by a factor of $m$. 
Consequently, it is reasonable to expect that using the averaged data 
$\hat{v}^{(m)}$ yields a more accurate approximation of $v$ 
then relying on any single observation $v_i$.

Next, we introduce a heuristic stopping rule for the \texttt{E-IRMGL+$\Psi$} method \eqref{main iterative scheme} that does not rely on the noise level but instead utilizes empirical information derived from the measurement data. The proposed rule is motivated by the ideas in \cite{hanke1996general, jin2016hanke,zhang2018heuristic}.
\begin{rulethm}[Heuristic principle]\label{rule:heuristic_discrepancy}
Let $\varrho \geq 1$ and let 
\[\Omega(k, \hat v^{(m)}) \coloneq (k + \varrho) \left\|\F(u_k^{(m)}) - \hat v^{(m)} \right\|^2.\]
Define $k_m^* \coloneq k_m^*(\hat{v}^{(m)})$ is an integer satisfying
\[k_m^* \in \arg\min \left \{ \Omega(k, \hat v^{(m)})\; \colon\;  k \in [0, k_\infty]  \right\},\]
where $k_\infty \coloneq  k_\infty(\hat v^{(m)})$ is the largest integer such that $u_k^{(m)} \in \mathcal{D}(\mathcal{F})$ for all $k \in [0, k_\infty].$ 
\end{rulethm}
Bakushinskii’s veto~\cite{bakushinskii1985remarks} asserts that heuristic stopping rules cannot guarantee convergence in the worst-case setting for any regularization method. Nevertheless, under additional assumptions on the noise, convergence of heuristic rules has been established for several regularization approaches~\cite{bajpai2024hanke,fu2020analysis,hanke1996general, jin2010heuristic,jin2016hanke, jin2017heuristic, liu2020heuristic, real2024hanke, zhang2018heuristic}. Motivated by the studies~\cite{jin2016hanke,zhang2018heuristic}, we adopt a similar condition in the context of measurement data and formulate the following assumption.
\begin{assumption}\label{assump:empirical}
There exists a constant $\kappa > 0$, independent of $m$, such that with high probability
\[
\|\hat v^{(m)} - \F(u)\| \;\geq\; \kappa \, \|\hat v^{(m)} - v\|
\]
for every $u \in S(\hat v^{(m)}) := \{\, u^{(m)}_k : k \in [0, k_\infty] \,\}$, where $\{u^{(m)}_k\}$ is the sequence generated by \texttt{E-IRMGL+}$\Psi$ \eqref{main iterative scheme}.
\end{assumption}
\begin{remark}
The above assumption guarantees that the residual 
$\|\hat v^{(m)} - \F(u)\|$ does not fall significantly below 
$\|\hat v^{(m)} - v\|$, thereby preventing degenerate noise realizations. 
Although this condition cannot be verified in practice, it is essential to circumvent Bakushinskii’s veto~\cite{bakushinskii1985remarks} and is standard in the analysis of heuristic parameter choice rules~\cite{bajpai2024hanke,jin2010heuristic,jin2017heuristic,kindermann2020convergence,real2024hanke}. 
Its practical validity is typically supported by numerical evidence; see~\cite{hubmer2022numerical} for a detailed study in the context of nonlinear Landweber iteration.
\end{remark}
To address the shortcomings of Assumption~\ref{assump:empirical}, we also introduce a stopping rule that leverages the average over multiple unbiased measurements. Recall that
\[
\mathbb{E}\bigl[\|\hat{v}^{(m)} - v\|^2\bigr] = \frac{\sigma^2}{m}.
\]
To estimate $\sigma^2$, we consider the empirical variance
\begin{equation}\label{eqn:empirial_varience}
    z_m^2 := \frac{1}{m-1} \sum_{i=1}^m \|v_i - \hat{v}^{(m)}\|^2,
\end{equation}
which satisfies $\mathbb{E}[z_m^2] = \sigma^2$ (cf. Assumption~\ref{assump:iid}). 
Thus, $z_m^2$ serves as an unbiased estimator of the variance $\sigma^2$, a 
standard fact in statistics (see also \cite{harrach2020beyond, jin2023dual}). Consequently, we may 
use $z_m / \sqrt{m}$ as an estimator of $\|\hat{v}^{(m)} - v\|$. This 
motivates the following statistical variant of the discrepancy principle, see~\cite{jin2023dual,yu2025landweber}.

\begin{rulethm}[Statistical discrepancy principle]\label{rule:statistical_discrepancy}
Define $k_m$ as 
\begin{equation}\label{eq:stopping_rule}
  k_m:=\min\Big\{\,0\le k\le K_{\max}(m)\;:\;
        \big\|\F\big(u^{(m)}_k\big)-\hat v^{(m)}\big\|
        \le \tfrac{\tau_m}{\sqrt m}\,z_m \Big\},
\end{equation}
where $K_{\max}(m)\in\mathbb N$ is a deterministic iteration budget satisfying $K_{\max}(m)\to\infty$ as $m\to \infty$ with the convention $k_m:=K_{\max}(m)$ when the set in \eqref{eq:stopping_rule} is
empty, and $\tau_m > 1$ may depend on $m$. 
\end{rulethm}
\begin{remark}
A priori bounds on the iteration number have previously been employed in the analysis of finite-dimensional approximations of iterative regularization methods, see~\cite{jin2001discrete, plato1990regularization}. Similar ideas have also been considered in the statistical inverse problems~\cite{blanchard2012discrepancy, harrach2020beyond, jin2013oracle,jin2023dual,lu2014discrepancy}. However, in those works the prescribed iteration limit primarily serves as a safeguard against excessive iterations that may deteriorate the reconstruction due to the influence of noise. In contrast, the present iteration budget is introduced to facilitate the convergence analysis of the proposed stopping strategy.
\end{remark}
It is worth emphasizing that, unlike the deterministic stopping index in the classical discrepancy principle~\cite{morozov1966solution} and the Hanke–Raus heuristic rule~\cite{real2024hanke}, the stopping indices $k_m^*$ and $k_m$ determined by \cref{rule:heuristic_discrepancy} and \cref{rule:statistical_discrepancy} are inherently \emph{random variables}. Consequently, their analysis necessitates the use of statistical tools.
So, for the convergence analysis of \texttt{E-IRMGL+$\Psi$} \eqref{main iterative scheme} under these two stopping rules, a key step is the construction of an event $\Gamma_m$ satisfying $\mathbb{P}(\Gamma_m) \to 1$ as $m \to \infty$, in line with the arguments of \cite{harrach2020beyond,jin2023dual}.  
For technical reasons, the subsequent regularization analysis requires that  
\[
\tau_m \to \infty \quad \text{and} \quad \frac{\tau_m}{\sqrt{m}} \to 0 \quad \text{as } m \to \infty.
\]  
Accordingly, we define  
\[
\Gamma_m := \Bigl\{\, |z_m - \sigma| \leq \tfrac{\sigma}{2}, \;\; \|\hat v^{(m)} - v\| \leq \tfrac{\tau_m}{\mathcal{H}\sqrt{m}}z_m \Bigr\},
\]  
where  
\[
\mathcal{H} > \frac{1+ \eta}{1 - \eta}
\]
with $\eta \in [0,1)$.
On this event, the subsequent convergence analysis will be carried out. 
With this construction, one can obtain  
\[
\mathbb{P}(\Gamma_m) \;\to\; 1 \quad \text{as } m \to \infty.
\]
Note that $k_m$ is a well-defined random variable satisfying
$0\le k_m\le K_{\max}(m)$ pointwise on the whole sample space. In
particular the approximate solution $u^{(m)}_{k_m}$ is defined also on $\Gamma_m^{c}$,
where the discrepancy criterion may never be met. The growth of $K_{\max}(m)$
is calibrated in \eqref{eq:budget} below. 
\section{Convergence analysis under statistical discrepancy principle}\label{sec:IRMGP} 
In this section, we aim to establish the regularization property of the proposed \texttt{E-IRMGL+\(\Psi\)} method with multiple repeated measurements, when terminated using \cref{rule:statistical_discrepancy}.
\subsection{The method}\label{subsec: The method}
We begin by providing a detailed exposition of the method.
To this end, we introduce the step size $\alpha_k^{(m)}$ and the weight parameter $\beta_k^{(m)}$, which are designed to enhance the speed of convergence of~\eqref{main iterative scheme}. A natural strategy is to select $\alpha_k^{(m)}$ and $\beta_k^{(m)}$ at each iteration so that the generated sequence remains sufficiently close to a solution of~\eqref{Model eqn}. Based on this principle, we define the parameters as
\begin{equation}\label{alpha}
 0 < \zeta \leq  \alpha_k^{(m)} = 
       \min \left\{ \frac{\zeta
    _0\|\F(u_k^{(m)}) - \hat v^{(m)}\|}{\|\F'(u_k^{(m)})^* (\F(u_k^{(m)}) -\hat v^{(m)})\|}, \zeta_1 \right\}, 
\end{equation}
\begin{equation}\label{beta}
   \beta_k^{(m)} =  \begin{cases}
      \min \left\{ \frac{\nu_0 \|\F (u_k^{(m)}) - \hat v^{(m)}\|^2}{\left\|\Delta_{u_k^{(m)}} u_k^{(m)}\right\|}, \frac{\nu_1}{\left\|\Delta_{u_k^{(m)}} u_k^{(m)}\right\|}, \nu_2 \right\} , & \text{if } \left\|\Delta_{u_k^{(m)}} u_k^{(m)}\right\| \neq 0\\ 
      0, &\text{if } \left\|\Delta_{u_k^{(m)}} u_k^{(m)}\right\| = 0,
    \end{cases} 
\end{equation}
where \(\zeta_0, \zeta_1, \nu_0, \nu_1\) and \(\nu_2\) denotes the fixed positive constants. We impose the following condition
\[
\zeta > \frac{\nu_0(\wp + \nu_1) + \zeta_0^2}{1 - \eta - \frac{1 + \eta}{ \mathcal H}}.
\]
This choice will later play a crucial role in ensuring the positivity of the constant introduced in \eqref{assum: On C}.
The pseudo-code for the iterative scheme~\eqref{main iterative scheme}, when stopped using \cref{rule:statistical_discrepancy}, is given in 
\Cref{alg:buildtree}, where at each iteration the graph Laplacian 
\(\Delta_{u_k^{(m)}}\) is reconstructed from the current iterate \(u_k^{(m)}\), 
ensuring adaptive updating of the regularization term throughout the process.

\begin{algorithm}
\caption{\texttt{E-IRMGL+\(\Psi\)} with statistical discrepancy principle}
\label{alg:buildtree}
\begin{algorithmic}[1]
\State \textbf{Input:} Operator $\F$, measurements $v_1, v_2, \ldots v_m,$ initial reconstructor $\Psi_{\theta_m}$, parameters $ \theta_m, \lambda, R,$ and $\tau_m>1$.
\State \textbf{Set:} $\hat v^{(m)} \coloneq \frac{1}{m} \sum_{i=1}^{m}v_i$ and $z_m \coloneq \sqrt{\frac{1}{m-1} \sum_{i=1}^{m} \|v_i - \hat v^{(m)}\|^2}.$
\State \textbf{Initialize:} $u_0^{(m)} \coloneq \Psi_{\theta_m}(\hat{v}^{(m)})$.
\While{$\|\F(u_k^{(m)}) - \hat{v}^{(m)}\| > \tfrac{\tau_m}{\sqrt{m}}z_m$}
    \State Compute $\Delta_{u_k^{(m)}}$ from $u_k^{(m)}, R$, and $\lambda$ as defined in \eqref{Eqn: Graph Laplacian}.
    \State Update the iterate
    \[
       u_{k+1}^{(m)} \coloneq u_k^{(m)} 
       - \alpha_k^{(m)} \F'(u_k^{(m)})^{*}\!\left(\F(u_k^{(m)}) - \hat v^{(m)}\right)
       - \beta_k^{(m)} \Delta_{u_k^{(m)}}u_k^{(m)},
    \]
    where $\alpha_k^{(m)}$ and $\beta_k^{(m)}$ are given by~\eqref{alpha} and~\eqref{beta}, respectively.
\EndWhile
\State \textbf{Output:} $u_{k_m}^{(m)}$.
\end{algorithmic}
\end{algorithm}

\noindent The following result establishes key properties of the iterates from \Cref{alg:buildtree}, which are crucial for proving finite termination of the scheme.  
\begin{lemma}\label{Lemma: Monotonicity}
Suppose that Assumptions~\ref{assump:Fi} and~\ref{assump:iid} hold. 
Let the sequence $\{u_k^{(m)}\}_{k \geq 0}$ be generated by Algorithm~\ref{alg:buildtree}, 
and let $k_m$ denote the stopping index defined by the statistical discrepancy principle 
(see~\cref{rule:statistical_discrepancy}). 
Assume that
\begin{equation}\label{assum: On C}
   \mathcal{C} \coloneq \Big( 1-\eta - \tfrac{1+\eta}{\mathcal{H}} \Big)\zeta 
   - \nu_0(\wp + \nu_1) - \zeta_0^2 > 0.
\end{equation}
For any integer $n \leq k_m$, the following statements hold on the event $\Gamma_m$:
\begin{enumerate}
    \item[(i)] \textbf{Monotonicity:} For all $0 \leq k < n$,
    \[
       \left\|u_{k+1}^{(m)} - u^\dagger\right\| \leq \left\|u_k^{(m)} - u^\dagger\right\|.
    \]
    \item[(ii)] \textbf{Stability and residual control:} 
    The iterates satisfy $u_k^{(m)} \in \mathcal{B}_\wp(u^\dagger)$ for all $0 \leq k \leq n$, and
    \[
       \sum_{k=0}^{n} \left\|\F(u_k^{(m)}) - \hat v^{(m)}\right\|^2 
       \leq \frac{1}{2\mathcal C}\,\left\|u_0^{(m)} - u^\dagger\right\|^2 < \infty.
    \]
\end{enumerate}
\end{lemma} 
\begin{proof}
By condition (A5) in Assumption~\ref{assump:Fi}, the initial iterate satisfies 
$u_0^{(m)} \in \mathcal{B}_\wp(u^\dagger)$. Assume further that 
$u_k^{(m)} \in \mathcal{B}_\wp(u^\dagger)$ for some iteration index $k$. 
It then follows that $u_k^{(m)} \in \mathcal{B}_{2\wp}(u_0^{(m)})$. 
We now define
\[
g_k^{(m)} := u_k^{(m)} - u^\dagger \quad \text{and} \quad h_k^{(m)} := \alpha_k^{(m)} \F'(u_k^{(m)})^*(\mathcal{F}(u_k^{(m)}) - \hat{v}^{(m)}) + \beta_k^{(m)} \Delta_{u_k^{(m)}}u_k^{(m)}.
\]
With these definitions in place, and by applying the update rule 
from \Cref{alg:buildtree}, we obtain
\begin{equation}\label{split eqn}
    \|g_{k+1}^{(m)}\|^2 - \|g_k^{(m)}\|^2 = \|h_k^{(m)}\|^2 - 2\langle g_k^{(m)}, h_k^{(m)} \rangle.
\end{equation}
Our next step is to examine the two terms on the right-hand side of \eqref{split eqn}. To begin, we consider the first term. By using the definitions of $\alpha_k^{(m)}$ and $ \beta_k^{(m)}$, we obtain 
\begin{align}\label{dk}
 \nonumber   \|h_k^{(m)}\|^2 &= \| \alpha_k^{(m)} \F'(u_k^{(m)})^*(\F(u_k^{(m)}) - \hat{v}^{(m)}) + \beta_k^{(m)} \Delta_{u_k^{(m)}}u_k^{(m)}\|^2 \\ \nonumber
    & \leq 2 \zeta_0^2 \|\F(u_k^{(m)}) - \hat{v}^{(m)}\|^2 + 2\nu_0 \nu_1\|\F(u_k^{(m)}) - \hat{v}^{(m)}\|^2 \\ 
    & \leq 2(\zeta_0^2 + \nu_0 \nu_1)\|\F(u_k^{(m)}) - \hat{v}^{(m)}\|^2. 
\end{align}
Similarly, by utilizing (A3) of Assumption~\ref{assump:Fi}, we have the estimate for the second term as
\begin{align}\label{ek, dk}
 \nonumber  -\langle g_k^{(m)}, h_k^{(m)} \rangle &= \alpha_k^{(m)} \langle \F'(u_k^{(m)})(u^\dagger - u_k^{(m)}), \F(u_k^{(m)}) - \hat v^{(m)} \rangle - \beta_k^{(m)} \langle g_k^{(m)}, \Delta_{u_k^{(m)}}u_k^{(m)} \rangle \\ \nonumber
 & = \alpha_k^{(m)} \langle \F(u_k^{(m)}) - \hat v^{(m)} +  \F'(u_k^{(m)})(u^\dagger - u_k^{(m)}), \F(u_k^{(m)}) - \hat v^{(m)} \rangle \\ \nonumber
 & \quad - \alpha_k^{(m)} \langle \F(u_k^{(m)}) - \hat v^{(m)}, \F(u_k^{(m)}) - \hat v^{(m)} \rangle + \beta_k^{(m)}\wp \| \Delta_{u_k^{(m)}}u_k^{(m)}\|  \\ \nonumber
   & =  \alpha_k^{(m)} \|\F(u_k^{(m)}) - \hat v^{(m)}\|\left[ \eta \|\F(u_k^{(m)}) - \hat v^{(m)}\| + (1 + \eta)\|v - \hat v^{(m)}\|\right] \\ \nonumber
   & \quad -  \alpha_k^{(m)} \|\F(u_k^{(m)}) - \hat v^{(m)}\|^2 +  \beta_k^{(m)}\wp \| \Delta_{u_k^{(m)}}u_k^{(m)}\|   \\ \nonumber
   & \leq - (1 - \eta)\alpha_k^{(m)}\|\F(u_k^{(m)}) - \hat v^{(m)}\|^2 + \wp\nu_0 \|\F(u_k^{(m)}) -\hat  v^{(m)}\|^2 \\
   \nonumber 
   & \quad + (1 + \eta)\alpha_k^{(m)}\|v - \hat v^{(m)}\|\|\F(u_k^{(m)}) - \hat v^{(m)}\|.
\end{align}
In deriving the last inequality, we have made use of \eqref{beta}. Since  $k < n \leq k_m$, it follows that
\[
\|\F(u_k^{(m)}) - \hat v^{(m)}\| >\frac{\tau_m}{\sqrt{m}} z_m.
\]
Furthermore, from the properties of the space $\Gamma_m$, we obtain
\[
\|v - \hat v^{(m)}\| \leq \frac{\tau_m}{\mathcal{H}\sqrt{m}} z_m.
\]
Combining these observations yields the bound
\begin{equation}\label{eqn: gk and hk}
        - \langle g_k^{(m)}, h_k^{(m)} \rangle 
 \leq - \left(  1 - \eta - \frac{1 + \eta}{\mathcal{H}} \right)
   \alpha_k^{(m)} \|\F(u_k^{(m)}) - \hat v^{(m)}\|^2
   + \wp \nu_0\|\F(u_k^{(m)}) - \hat v^{(m)}\|^2.
\end{equation}
By substituting \eqref{dk} and \eqref{eqn: gk and hk} into (\ref{split eqn}), we obtain
\begin{equation}\label{Concluding eqn of first Lemma}
         \|g_{k+1}^{(m)}\|^2 - \|g_k^{(m)}\|^2 \leq -2\left[ \left( 1 - \eta - \frac{1 + \eta}{\mathcal{H}}\right)\zeta -\nu_0(\wp + \nu_1) - \zeta_0^2\right]
     \|\F(u_k^{(m)}) - \hat v^{(m)}\|^2.
\end{equation}
Finally, substituting (\ref{assum: On C}) into (\ref{Concluding eqn of first Lemma}) yields
\begin{equation}\label{Concluding eqn of first Lemma_II}
         \|g_{k+1}^{(m)}\|^2 - \|g_k^{(m)}\|^2 \leq -2\mathcal{C}\|\F(u_k^{(m)}) - \hat v^{(m)}\|^2.
\end{equation}
Thus, the monotonicity holds for $0 \leq k <n$.

\noindent Next, we turn to the proof of assertion $(ii)$. The inequality (\ref{Concluding eqn of first Lemma_II}) ensures that
\[
\|u_{k+1}^{(m)} - u^\dagger\| \leq \|u_{k}^{(m)} - u^\dagger\| \leq \wp,
\]
which suggests that $u_{k+1}^{(m)} \in \mathcal{B}_\wp(u^\dagger)$. From $k = 0$ to $k = n$, the inequality (\ref{Concluding eqn of first Lemma_II}) can be summed to produce
\begin{equation}\label{Concluding eqn of first Lemma in C}
   \sum_{k=0}^{n} \left[\|g_{k+1}^{(m)}\|^2 - \|g_k^{(m)}\|^2\right] \leq -2\mathcal{C}  \sum_{k=0}^{n}\|\F(u_k^{(m)}) - \hat v^{(m)}\|^2.
\end{equation}
In conclusion, \eqref{Concluding eqn of first Lemma in C} implies that
\begin{equation*}
 \sum_{k=0}^{n}\|\F(u_k^{(m)}) -\hat v^{(m)}\|^2 \leq \frac{1}{2\mathcal{C}} \|g_{0}^{(m)}\|^2 = \frac{1}{2\mathcal{C}}\|u_0^{(m)} - u^\dagger\|^2.
\end{equation*}
 This establishes the desired result.
\end{proof}
\noindent The next result, shows the finite iteration termination property of  Algorithm \ref{alg:buildtree}. 
\begin{theorem}
  Given that the conditions of Lemma~\ref{Lemma: Monotonicity} hold, and \Cref{alg:buildtree} is initialized with a starting point $u_0^{(m)}$. Then the algorithm exhibits finitely many iterations, i.e., there exist a random integer $k_m < \infty$ such that 
    \begin{equation*}
      \|\F(u^{(m)}_{k_m}) - \hat{v}^{(m)}\| \leq \frac{\tau_m}{\sqrt{m}} z_m < \|\F(u_k^{(m)}) - \hat v^{(m)}\|, \quad 0 \leq k < k_m \qquad \text{on } \Gamma_m.
    \end{equation*}
\end{theorem}
\begin{proof}
Le \( n \) be a non-negative integer such that
\[ \|\F (u_k^{(m)}) - \hat v^{(m)}\| > \frac{\tau_m}{\sqrt{m}} z_m \quad \forall \; k \in \{ 0, 1, \dots, n\}. \] 
By invoking \eqref{Concluding eqn of first Lemma in C}, yields
\begin{align*}
   2(n+1)\mathcal C\frac{\tau_m^2}{m}z_m^2 \leq  2\mathcal C\sum_{k=0}^{k=n}  \|\F(u_k^{(m)}) - \hat v^{(m)}\|^2 &\leq \sum_{k=0}^{k = n}\left[\|g_{k}^{(m)}\|^2 - \|g_{k+1}^{(m)}\|^2\right]\\
     & = \|g_0^{(m)}\|^2 - \|g_{n+1}^{(m)}\|^2 \leq \|g_0^{(m)}\|^2.
\end{align*}
Using the estimate $z_m \geq \sigma/2,$ which holds on $\Gamma_m,$ we obtain 
\begin{equation*}
      (n+1)\mathcal{C} \frac{\tau_m^2}{2m} \sigma^2 \leq \|g_0^{(m)}\|^2 < \infty.
\end{equation*}
If no finite index $k_m$ satisfied the stopping condition \eqref{eq:stopping_rule}, then letting $n \to \infty$ in the above estimate would lead to a contradiction. Hence, \Cref{alg:buildtree} must terminate after finitely many iterations.
\end{proof}
\subsection{Convergence for exact data}\label{Subsec: Convergence for exact data}
We begin by examining Algorithm~\ref{alg:buildtree} for exact data $v$, suppressing the superscript $(m)$ throughout this subsection.  Algorithm~\ref{alg:buildtree_exact} details the corresponding computational scheme. To establish the convergence of the proposed method, it is sufficient to show that the iterate sequence $\{u_k\}_{k \geq 0}$ produced by Algorithm~\ref{alg:buildtree_exact} is Cauchy. To this end, we first establish a key preliminary lemma regarding the sequence $\{u_k\}_{k \geq 0}$.
\begin{algorithm}
\caption{\texttt{E-IRMGL+\(\Psi\)} with exact data}
\label{alg:buildtree_exact}
\begin{algorithmic}[H]
\State \textbf{Input:} Operator $\F$, exact data $v = \F(u^\dagger)$, initial reconstructor $\Psi_\theta$, parameters $\theta, \lambda, R$.
\State \textbf{Initialize:} $u_0 \coloneq \Psi_\theta(v)$.
\For{$k = 0, 1, 2, \ldots$}
    \State Compute $\Delta_{u_k}$ from $u_k, R$, and $\lambda$ as defined in
      \eqref{Eqn: Graph Laplacian}.
    \State Update the iterate
    \[
       u_{k+1} \coloneq u_k
       - \alpha_k \F'(u_k)^{*}\!\left(\F(u_k) - v\right)
       - \beta_k \Delta_{u_k}u_k ,
    \]
    where $\alpha_k$ and $\beta_k$ are given by
  \begin{equation*}
  \begin{split}
    \alpha_k &= 
       \min \left\{ \frac{\zeta
    _0\|\F(u_k) - v\|}{\|\F'(u_k)^* (\F(u_k) - v)\|}, \zeta_1 \right\} \quad \text{when } \|\F'(u_k)^* (\F(u_k) - v)\| \neq 0, \\
   \beta_k &=  \begin{cases}
      \min \left\{ \frac{\nu_0 \|\F (u_k) - v\|^2}{\left\|\Delta_{u_k} u_k\right\|}, \frac{\nu_1}{\left\|\Delta_{u_k} u_k\right\|}, \nu_2 \right\} , & \text{if } \left\|\Delta_{u_k} u_k\right\| \neq 0\\ 
      0, &\text{if } \left\|\Delta_{u_k} u_k\right\| = 0.
    \end{cases}   
  \end{split}
\end{equation*}
\EndFor
\State \textbf{Output:} the sequence $\{u_k\}_{k \geq 0}$.
\end{algorithmic}
\end{algorithm}
\begin{lemma}\label{lemma: monotonicity for exact data}
    Suppose that Assumption~\ref{assump:Fi} hold, and let $\{u_k\}_{k \geq 0}$ denote the sequence generated by \Cref{alg:buildtree_exact}. Then $u_k \in \mathcal{B}_{\wp}(u^\dagger)$ for all $k \geq 0$, and 
    \begin{equation}\label{eqn:used_in_last_thm}
        \|u_{k+1} - u^\dagger\|^2 - \|u_k - u^\dagger\|^2 
        \leq -2\mathcal{C}_0 \|\F(u_k) - v\|^2, 
        \quad \forall \; k \geq 0,
    \end{equation}
    where
    \(
    \mathcal{C}_0 := (1 - \eta)\zeta - \nu_0(\wp + \nu_1) - \zeta_0^2 > 0.
    \)
    Consequently, the sequence $\{\|u_k - u^\dagger\|\}_{k \geq 0}$ decreases monotonically, and
    \begin{equation}\label{sum for non noisy case}
        \sum_{k=0}^\infty \|\F(u_k) - v\|^2 
        \leq \frac{1}{2\mathcal{C}_0} \|u_0 - u^\dagger\|^2 < \infty,
    \end{equation}
    where $u_0:=\Psi_\theta(v).$
\end{lemma}
\begin{proof}
 The assertion can be established through reasoning analogous to that employed in Lemma~\ref{Lemma: Monotonicity}.
\end{proof}
\noindent From \eqref{sum for non noisy case}, we deduce that for exact data the residual vanishes, i.e., $\|\F(u_k)-v\|\to 0$ as $k\to\infty$. This allows us to state the convergence result in the exact data setting.
\begin{theorem}\label{Convergence for exact data}
Assuming the conditions of Lemma~\ref{lemma: monotonicity for exact data} are satisfied, the sequence of iterates $\{u_k\}_{k \geq 0}$ produced by  \Cref{alg:buildtree_exact} converges to a solution, denoted $u^\dagger$, of the equation (\ref{Model eqn}).
\end{theorem}
\begin{proof}
For the exact data setting, define \( g_k := u_k - u^\dagger \).  
Given indices \( \ell \geq k \), select an integer \( m \) with \( \ell \geq m \geq k \) in a way that  
\begin{equation}\label{rn rp ineq}
    \|\F(u_m) - v\| \leq \|\F(u_n) - v\|, \quad \forall \; k \leq n \leq \ell .
\end{equation}
By the triangle inequality, we obtain  
\begin{equation}\label{xk triangular}
    \|g_\ell - g_k\| \leq \|g_\ell - g_m\| + \|g_m - g_k\|,
\end{equation}    
where the two terms admit the following expansions:
\begin{align}
    \|g_\ell - g_m\|^2 &= 2\langle g_m - g_\ell, g_m \rangle + \|g_\ell\|^2 - \|g_m\|^2,  \label{gm gl} \\
    \|g_m - g_k\|^2 &= 2\langle g_k - g_m, g_m \rangle + \|g_m\|^2 - \|g_k\|^2. \label{gm gk}
\end{align}
From Lemma~\ref{lemma: monotonicity for exact data}, the sequence \(\{\|g_k\|\}_{k \geq 0}\) is monotonically decreasing and bounded below by \(0\). Hence, the limit exists, say  
\(
\lim_{k \to \infty} \|g_k\| = \Theta \geq 0.
\)  
Consequently, we deduce
\begin{align}
     \lim_{k \to \infty }\|g_\ell - g_m\|^2 &= 2 \lim_{k \to \infty }\langle g_m - g_\ell, g_m \rangle + \Theta^2 - \Theta^2, \label{xl xm limit}\\
     \lim_{k \to \infty }\|g_m - g_k\|^2 &= 2 \lim_{k \to \infty }\langle g_k - g_m, g_m \rangle + \Theta^2 - \Theta^2. \label{xm xk limit}
\end{align}
We want to show that $\{g_k\}_{k \geq 0}$ is a Cauchy sequence. Therefore, we assert that as $k \to \infty$, $\langle g_m - g_\ell, g_m \rangle \to 0$ and $\langle g_k - g_m, g_m \rangle \to 0$. In order to substantiate the assertion, it is evident that  
\begin{align}
 \nonumber   |\langle g_m &- g_\ell, g_m\rangle| = |\langle u_m - u_\ell, g_m\rangle| = \left| \sum_{j=m}^{\ell -1} \langle u_{j+1} - u_j, g_m \rangle \right| \\ \nonumber
    & \leq \left| \left\langle \sum_{j=m}^{\ell-1} \left(\alpha_j \F'(u_j)^*(\F(u_j) - v) + \beta_j \Delta_{u_j}u_j \right), u_m - u^\dagger \right\rangle \right| \\ \nonumber
     & \leq \sum_{j=m}^{\ell-1}\left| \left\langle  \alpha_j \F'(u_j)^*(\F(u_j) - v), u_m -u^\dagger \right\rangle \right| + \sum_{j=m}^{\ell -1}\left| \left\langle \beta_j \Delta_{u_j}u_j, u_m - u^\dagger \right\rangle \right| \\ \nonumber
     & \leq \zeta_1 \sum_{j=m}^{\ell-1}\left| \left\langle   \F(u_j)- v, \F'(u_j)(u_m - u^\dagger) \right\rangle \right| + \sum_{j=m}^{\ell-1}\beta_j \left| \left\langle  \Delta_{u_j}u_j, u_m - u^\dagger \right\rangle \right| \\ \nonumber
     & \leq \zeta_1 \sum_{j=m}^{\ell-1}\left \|  \F(u_j)- v \right\| \left\| \F'(u_j)(u_m - u^\dagger ) \right\| + \sum_{j=m}^{\ell-1}\beta_j \left\| \Delta_{u_j}u_j\right\|  \left \|u_m - u^\dagger \right\|.
\end{align}
Using (A3) of Assumption~\ref{assump:Fi}, it can be easily obtain that 
\begin{align}\label{eqn:just_above}
 \|\F'(u_j)(u_m - u^\dagger) \|   &\leq  \|\F'(u_j)(u_j - u^\dagger) \| +  \|\F'(u_j)(u_m - u_j) \| \\ \nonumber
 & \leq (1+\eta) \left( \|\F(u_j) - v\| + \|\F(u_m) - \F(u_j)\| \right) \\ \nonumber
 & \leq (1+\eta) \left( 2\|\F(u_j) - v\| + \|\F(u_m) - v\| \right).
\end{align}
Employing \eqref{rn rp ineq}, \eqref{eqn:just_above}, and the inequality $\beta_j \left\|\Delta_{u_j}u_j\right\| \leq \nu_0 \|\F(u_j) - v\|^2$, from \eqref{beta}, and observing that $u_k \in \mathcal{B}_\wp(u^\dagger)$ for all $k \geq 0$, the preceding inequality can be equivalently transformed as
\begin{align}\label{x_n g_l in terms of residual}
   \nonumber   |\langle g_m - g_\ell, g_m\rangle| & \leq 3\zeta_1(1+\eta) \sum_{j=m}^{\ell-1}\|\F(u_j) - v\|^2 + \wp \nu_0 \sum_{j=m}^{\ell-1}\|\F(u_j) - v\|^2 \\ 
   &  \leq \Big( 3\zeta_1 (1 + \eta) +  \wp\nu_0 \Big)\sum_{j=m}^{\ell-1}\|\F(u_j) - v\|^2.
\end{align}
Applying the same reasoning to $|\langle g_k - g_m, g_m \rangle|$, we can obtain
\begin{equation}\label{gk gm to 0}
 |\langle g_k - g_m, g_m \rangle| \leq    \Big( 3\zeta_1(1+\eta) + \wp \nu_0\Big)\sum_{j=k}^{m-1}\|\F(u_j) - v\|^2.
\end{equation}
Now, using \eqref{x_n g_l in terms of residual} and \eqref{gk gm to 0} we deduce that 
\( |\langle g_m - g_\ell, g_m \rangle| \) and 
\( |\langle g_k - g_m, g_m \rangle| \) converge to zero as 
\( k \to \infty \), by virtue of \eqref{sum for non noisy case}. 
So, from \eqref{xl xm limit} and \eqref{xm xk limit}, we have
\begin{align*}
     \lim_{k \to \infty }\|g_l - g_m\|^2 &= 2 \lim_{k \to \infty }\langle g_m - g_l, g_m \rangle =0, \\
     \lim_{k \to \infty }\|g_m - g_k\|^2 &= 2 \lim_{k \to \infty }\langle g_k - g_m, g_m \rangle =0.
\end{align*}
It is straightforward from \eqref{xk triangular}, \eqref{gm gl}, and \eqref{gm gk} that the sequence $\{g_k\}_{k \geq 0}$ is Cauchy. Therefore, $\{u_k\}_{k \geq 0}$ likewise forms a Cauchy sequence and converges to a limit $u^*$. Given that \(F(u_k) - v\) vanishes as \(k \to \infty \), the limit satisfies \eqref{Model eqn}, suggesting \(u^* = u^\dagger\). Thus, the iterates generated by exact data variant of \texttt{E-IRMGL+\(\Psi\)}~\eqref{main iterative scheme} converge to the true solution of \eqref{Model eqn}.
\end{proof}
\noindent As an immediate consequence of the preceding theorem and Algorithm~\ref{alg:buildtree_exact}, we obtain the following proposition.
\begin{proposition}\label{prop:stationary}
Let $\{u_k\}_{k\ge0}$ be the iterates produced by Algorithm~\ref{alg:buildtree_exact} and suppose that $\F(u_K)=v$ for some
$K\in\mathbb N$. Then, under the assumptions of Lemma~\ref{lemma: monotonicity for exact data},
\[
  u_k=u_K\quad\text{and}\quad\F(u_k)=v\qquad\text{for every }k\ge K ,
\]
and consequently $u^\dagger=\lim_{k\to\infty}u_k=u_K$, i.e.\ $u_K$ is the
solution of \eqref{Model eqn} produced by
Theorem~\ref{Convergence for exact data}.
\end{proposition}
\begin{proof}
We prove the claim by induction on $k \geq K$. The base case $k=K$ holds by assumption. Suppose that, for some $k \geq K$, we have
\(
u_k=u_K \) and 
\(\F(u_k)=v.\)
According to the update rule in Algorithm~\ref{alg:buildtree_exact},
\[
u_{k+1}
=u_k-\alpha_k\,\F'(u_k)^{*}\big(\F(u_k)-v\big)
-\beta_k\,\Delta_{u_k}u_k.
\]
Since $\F(u_k)=v$, the first term vanishes. Moreover, by the definition of the step size $\beta_k$ in Algorithm~\ref{alg:buildtree_exact}, we have $\beta_k=0$ whenever $\F(u_k)=v$. Consequently, the second term also vanishes. Therefore,
\(
u_{k+1}=u_k=u_K.
\)
It follows immediately that
\[
\F(u_{k+1})=\F(u_K)=v.
\]
Hence, the induction hypothesis holds for $k+1$, which completes the proof.
\end{proof}
\subsection{Stability analysis}\label{subsec: stability}
This subsection, investigates the stability of the \texttt{E-IRMGL+$\Psi$} approach in the presence of multiple repeated measurement data. 
We require the following regularity condition on the family of reconstructors $\Psi_\theta$.  

\begin{assumption}\label{hypothesis stability}
Let $\{\Psi_\theta\}$, $\Psi_\theta:\mathcal V\to\mathcal U$,
be the family of reconstruction operators and let
$\theta(\,\cdot\,,\cdot\,)$ be the parameter choice rule, so that the
initial iterates of Algorithm~\ref{alg:buildtree} and Algorithm~\ref{alg:buildtree_exact} are
\[
  u^{(m)}_0=\Psi_{\theta_m}\big(\hat v^{(m)}\big),
  \qquad \theta_m:=\theta\big(m,\hat v^{(m)}\big),
  \qquad\text{and}\qquad
  u_0=\Psi_{\theta}(v),
  \qquad \theta:=\theta(v) .
\]
We assume that the pair $(\Psi,\theta)$ is mean-square consistent at
$v$, i.e., whenever $\{y_m\}_{m\in\mathbb N}$ is a sequence of $\mathcal V$-valued
random elements with
\(
  \lim_{m\to\infty}\mathbb E\big[\|y_m-v\|^{2}\big]=0 ,
\)
the parameters $\theta_m':=\theta(m,y_m)$ satisfy
\begin{equation*}\label{eq:msc}
  \lim_{m\to\infty}\mathbb E\big[|\theta_m'-\theta|^{2}\big]=0
  \qquad\text{and}\qquad
  \lim_{m\to\infty}\mathbb E\Big[\big\|\Psi_{\theta_m'}(y_m)-\Psi_{\theta}(v)\big\|^{2}\Big]=0 .
\end{equation*}
\end{assumption}  
\noindent It is crucial to emphasize that the aforementioned assumption is the statistical counterpart of \cite[Hypothesis 3.2]{bianchi2025data}, and we are working on $\hat v^{(m)}$ instead of noisy data $v^\delta$. It should be noted that a wide class of reconstructors can satisfy this weak assumption. 
We verify Assumption~\ref{hypothesis stability} on a standard reconstructor given below and clarify the role and admissible dependence of the parameter $\theta_m=\theta(m,\hat v^{(m)})$.
\subsubsection*{Tikhonov reconstructors}
Define the Tikhonov functional
\[
\mathcal J_\theta(u;v):=\tfrac12 \left(\|\F(u)-v\|^2 + \theta\|u\|^2\right),
\qquad v\in\mathcal V,
\]
where $\theta \in (0, \infty).$
Suppose there exist a radius $\wp>0$ and a constant $c_\theta>0$ such that for every $u \in \mathcal B_\wp(u^\dagger)$  the Hessian of $\mathcal J_\theta$ at $u$ satisfies the uniform coercivity bound
\[
\bigl\langle \bigl(\F'(u)^*\F'(u)+\theta I\bigr) w, w\bigr\rangle \;\ge\; c_\theta \|w\|^2
\qquad\text{for all } w\in\mathcal U.
\]
Then for every $v$ sufficiently close to $\F(u^\dagger)$ the minimizer
\[
\Psi_\theta(v):=\arg\min_{u\in \mathcal B_\wp(u^\dagger)}\mathcal J_\theta(u;v)
\]
exists and is unique. Let $\theta = \theta(m, \hat v^{(m)})$ denotes the regularization parameter 
determined via the statistical version of the discrepancy principle 
\cite[equation (4.57)]{hanke1996general}, then the parameter sequence is 
stable in expectation, in the sense that
\[
\mathbb{E}[\theta_m] \;\to\; \mathbb{E}[\theta]
\qquad \text{as } m \to \infty.
\]
  Moreover the map $(v, \theta)\mapsto\Psi_\theta(v)$ is joint locally Lipschitz, i.e., there exists a constant $L_v, L_\theta>0$ (depending on $\theta$ and the local behavior of $\F$) such that
\[
\|\Psi_{\theta_m}(\hat v^{(m)})-\Psi_{\theta}(v)\|\le L_v\|\hat v^{(m)}-  v\| + L_\theta |\theta_m - \theta|
\]
for all $\hat v^{(m)},  v$ in a small neighborhood of $\F(u^\dagger)$.
Consequently, we have the desired result.
 In the linear setting, the verification of Assumption~\ref{hypothesis stability} becomes significantly simpler. One may follow the argument outlined in \cite[Example A.6]{bianchi2025data}, which provides a direct and concise proof tailored to the linear case.

Before proceeding to the main stability result, we note that the 
regularization term of \texttt{E-IRMGL+$\Psi$} method
depends explicitly on the empirical mean \(\hat v^{(m)}\). 
This dependence prevents a direct application of standard analytical 
techniques and making it highly nonstandard. To address this difficulty, we rely on a lemma, originally 
established in~\cite{bajpai2025convergence,bianchi2025data}, which is 
central to our stability analysis.
\begin{lemma}\label{lemma: contraction of Delta}
  For the sequence of iterates \(\{u_k^{(m)}\}_{k \geq 0}\)  generated by \Cref{alg:buildtree} and the sequence of iterates $\{u_k\}_{k \geq 0}$ generated by its exact data counterpart, the following holds
  \begin{equation*}
      \|\Delta_{u_k^{(m)}}u_k - \Delta_{u_k}u_k\| \leq \mathcal{H}_0 \|u_k\|\|u_k^{(m)} - u_k\|,
  \end{equation*}
 where 
 \[\mathcal{H}_0 =2 L_h \left( \max_{a,b}{g(a, b)}\sqrt D_N  +  \max_{a}\max_{b}{g(a, b)}\right),\] 
 with \(g(a,b)\) as in \eqref{eqn: edge weight fun}, $L_h$ is the Lipchitz constant of the function $h_u$ and $D_N$ denotes the maximum degree of any node in $\mathcal S$.
\end{lemma}
\noindent We now show that Lemma~\ref{lemma: contraction of Delta}, together with Assumption~\ref{hypothesis stability}, ensures the stability linking Algorithm~\ref{alg:buildtree} with its exact counterpart. Since the iterates $u_k^{(m)}$ of Algorithm~\ref{alg:buildtree} are random variables, the analysis is carried out in a stochastic framework.

 \begin{lemma}\label{lemma: stability}
Let Assumptions~\ref{assump:Fi}, \ref{assump:iid}, and \ref{hypothesis stability} be satisfied. Then, for every fixed integer $k \geq 0$, the iterates $u_k^{(m)}$ produced by Algorithm~\ref{alg:buildtree} converge in mean square to the corresponding exact-data iterates $u_k$, i.e.,
\[
    \lim_{m \to \infty} \mathbb{E}\!\left[\|u_k^{(m)} - u_k\|^2\right] = 0 .
\]
 \end{lemma}
 \begin{proof}
We proceed by induction. For $k=0$, since $u_0^{(m)} = \Psi_{\theta_m}(\hat v^{(m)})$ and $u_0 = \Psi_\theta(v)$, 
the result follows immediately from Assumption~\ref{hypothesis stability}, which ensures 
\[
\mathbb{E}\!\left[\|u_0^{(m)} - u_0\|^2\right] \to 0 \quad \text{as } m \to \infty.
\]
Let us now assume that the assertion is true for all $0 \leq j \leq k$, that is, as $m$ approaches $\infty$, \( \mathbb{E}[|u_j^{(m)} - u_j\|^2] \to 0\).
Our goal is to demonstrate that it also applies to $j=k+1$. To do so, observe from \eqref{main iterative scheme} that 
\begin{equation}\label{g_k+1 to g_k_exp}
\begin{aligned}
u_{k+1}^{(m)} &- u_{k+1} 
    = \left(u_k^{(m)} - u_k \right)  
       - \Big( \beta_k^{(m)} \Delta_{u_k^{(m)}} u_k^{(m)} - \beta_k \Delta_{u_k} u_k \Big) \\
    &\qquad - \Big( \alpha_k^{(m)} \F'(u_k^{(m)})^*(\F(u_k^{(m)}) - \hat v^{(m)}) 
          - \alpha_k \F'(u_k)^*(\F(u_k) - v) \Big).
\end{aligned}
\end{equation}
Taking the squared norm and subsequently the expectation yields
\begin{equation}\label{eq:exp_bound}
\begin{aligned}
&\mathbb{E}\!\left[\|u_{k+1}^{(m)} - u_{k+1}\|^2\right]
\leq 3 \Big( 
      \mathbb{E}\!\left[\|u_k^{(m)} - u_k\|^2\right] \\
&\quad + \mathbb{E}\!\left[\|\alpha_k^{(m)} \F'(u_k^{(m)})^*(\F(u_k^{(m)}) - \hat v^{(m)}) 
        - \alpha_k \F'(u_k)^*(\F(u_k) - v)\|^2\right] \\
&\quad + \mathbb{E}\!\left[\|\beta_k^{(m)} \Delta_{u_k^{(m)}} u_k^{(m)} 
        - \beta_k \Delta_{u_k} u_k\|^2\right] \Big) \coloneq  3(\text{I} + \text{II} + \text{III}),
\end{aligned}
\end{equation}
where we have employed the inequality $(a_1 + a_2 + a_3)^2 \leq 3 (a_1^2 + a_2^2 + a_3^2)$ for $a_1, a_2, a_3 \in \mathbb{R}$. Note that  I $\to 0$ by induction argument. We now turn to the estimate of II. For this, let
\[
\mathcal A_m := \F'(u_k^{(m)})^*\big(\F(u_k^{(m)})- \hat v^{(m)}\big),
\qquad 
\mathcal A := \F'(u_k)^*\big(\F(u_k)-v\big),
\]
and denote $\alpha_m := \alpha_k^{(m)}$, $\alpha := \alpha_k$. Then
\[
\alpha_m \mathcal  A_m - \alpha \mathcal  A = \alpha_m(\mathcal A_m- \mathcal A) + (\alpha_m-\alpha) \mathcal A.
\]
By the standard inequality $\|x+y\|^2 \leq 2\|x\|^2+2\|y\|^2$, we obtain
\begin{equation}\label{eq:decomp}
\mathbb{E}\left[\|\alpha_m \mathcal A_m - \alpha \mathcal  A\|^2\right]
\leq 2\,\mathbb{E}\!\left[\alpha_m^2 \| \mathcal A_m- \mathcal A\|^2\right]
   + 2\,\mathbb{E}\!\left[(\alpha_m-\alpha)^2 \|\mathcal A\|^2\right].
\end{equation} 
We now expand \(\mathcal{A}_m - \mathcal{A}\) as
\begin{align*}
 \mathcal{A}_m-\mathcal{A} & = \big(\F'(u_k^{(m)})^*-\F'(u_k)^*\big)\big(\F(u_k) - v\big) \\
& \quad + \left[ \F'(u_k^{(m)})^*(\F(u_k^{(m)}) - \hat v^{(m)} ) - \F'(u_k^{(m)})^*(\F(u_k)-v)\right].   
\end{align*}
By combining the operator norm inequality with the triangle inequality, and invoking both the uniform boundedness and the Lipschitz continuity of $\F'$ (cf. (A2) and (A4) of Assumption~\ref{assump:Fi}), we obtain
\begin{align*}
\| \mathcal A_m- \mathcal A\| \leq  L \|u_k^{(m)}- u_k \| \,\|\F(u_k)-v\| + B\|\F(u_k^{(m)})-\F(u_k) - \hat v^{(m)} + v\|.
\end{align*}
Furthermore, (A2) and (A3) of Assumption~\ref{assump:Fi} ensures that
\[\|\F(u_k) - v\| \leq \frac{1}{1 - \eta} \|\F'(u^\dagger)(u_k - u^\dagger)\| \leq \frac{B}{1 - \eta} \|u_k - u^\dagger \|.\]
Since Theorem~\ref{Convergence for exact data} guarantees that the sequence ${u_k}$ remains bounded, it follows that $\|\F(u_k) - v\|$ is also bounded. Taking squared norms and expectations then yields
\begin{align*}
\mathbb{E} \left[ \|\mathcal A_m- \mathcal A\|^2 \right] 
&\leq 2  \,\Big(L^2 \|\F(u_k)-v\|^2 \mathbb{E} \left[\|u_k^{(m)}-u_k\|^2 \right]  \\
&\qquad + B^2 \mathbb{E} \left [ \|\F(u_k^{(m)})-\F(u_k) - \hat v^{(m)} + v\|^2  \right]\Big)\\
&\leq 2  \,\Big(L^2 \|\F(u_k)-v\|^2 \mathbb{E} \left[\|u_k^{(m)}-u_k\|^2 \right]  \\
&\qquad + 2B^2 \left(\mathbb{E} \left [ \|\F(u_k^{(m)}) - \F(u_k) \|^2 \right] + \mathbb{E} \left [ \| \hat v^{(m)} - v\|^2  \right] \right)\Big) \\
&\leq 2  \,\Big(L^2 \|\F(u_k)-v\|^2 + \frac{2B^4}{(1-\eta)^2} \Big) \mathbb{E} \left[\|u_k^{(m)}-u_k\|^2 \right]+ 4B^2 \frac{\sigma^2}{m}.
\end{align*}
By assumption, we have 
\(
\mathbb{E} \left [ \|u_k^{(m)}-u_k \|^2 \right] \to 0 \) and
\( \frac{\sigma^2}{m} \to 0\) as \(m \to \infty. 
\)
Moreover, since $\alpha_m \leq \zeta_1$ by definition, it follows that 
\begin{equation}\label{eqn:A_m to A}
    \mathbb{E}\!\left[\alpha_m^2 \|\mathcal A_m- \mathcal A\|^2\right] 
    \leq \zeta_1^2 \,\mathbb{E}\!\left[\|\mathcal A_m- \mathcal A\|^2 \right] \to 0.
\end{equation}
Next, we establish that $\alpha_m \to \alpha$ in $\mathcal L^2(\Sigma, \mathcal{U})$, which directly implies
\[
\mathbb{E}\!\left[(\alpha_m-\alpha)^2\|\mathcal A\|^2\right]
= \|\mathcal A\|^2 \,\mathbb{E}\!\left[(\alpha_m-\alpha)^2\right] \to 0.
\]
Since $u_k^{(m)}\to u_k$ and $\hat v^{(m)}\to v$ as $m \to \infty$ in mean-square, Chebyshev's inequality gives, for any $\varepsilon>0$,
\[
\mathbb{P}\big(\|u_k^{(m)}-u_k\|>\varepsilon\big)
\le \frac{\mathbb{E}[\|u_k^{(m)}-u_k\|^2]}{\varepsilon^2}\to 0,
\]
and similarly for $\hat v^{(m)}$. Hence $(u_k^{(m)}, \hat v^{(m)})\to (u_k,v)$ in probability. Now, 
define
\[
\phi(u,v) := \frac{\zeta_0\|\F(u)-v\|}{\|\F'(u)^*(\F(u)-v)\|},
\]
with the understanding that $\phi$ is considered only in a neighborhood of $(u_k,v)$ where the denominator is nonzero. By (A4) of Assumption~\ref{assump:Fi}, the map $(u,v)\mapsto \phi(u,v)$ is continuous at $(u_k,v)$ because $\|\F'(u_k)^*(\F(u_k)-v)\|=\|\mathcal A\|\neq 0$ implies the denominator stays bounded away from zero in a small neighborhood. So, the function
\[
\psi(u,v):=\min\{\phi(u,v),\zeta_1\}
\]
is therefore also continuous at $(u_k,v)$. Since $\alpha_m=\psi(u_k^{(m)}, \hat v^{(m)})$ and $\alpha=\psi(u_k,v)$, the continuous mapping theorem yields
\(
\alpha_m \xrightarrow{\;\mathbb{P}\;} \alpha\) (convergence in probability).
By definition $\;0\le \alpha_m\le \zeta_1\;$ and $\;0\le \alpha\le\zeta_1\,$, hence
\[
0\le Z_m:=|\alpha_m-\alpha|^2 \le 4\zeta_1^2<\infty \quad\text{a.s.}
\]
We already have $Z_m \xrightarrow{\;\mathbb{P}\;} 0$. We claim $ \mathbb{E}[Z_m]\to 0$.  
To prove this, assume for contradiction that $\mathbb{E}[Z_m]\not\to 0$. Then there exists $\varepsilon>0$ and a subsequence $(Z_{m_j})$ with $\mathbb{E}[Z_{m_j}]\ge\varepsilon$ for all $j$. Since $Z_{m_j}\to 0$ in probability, by the subsequence principle there exists a further sub-subsequence $(Z_{m_{j_\ell}})$ that converges to $0$ almost surely. Because $0\le Z_{m_{j_\ell}}\le 4 \zeta_1^2$ and $4\zeta_1^2$ is integrable, the dominated convergence theorem implies that
\(
\mathbb{E}[Z_{m_{j_\ell}}]\to 0,
\)
which contradicts $\mathbb{E}[Z_{m_j}]\ge\varepsilon$. Thus no such subsequence exists and we must have $\mathbb{E}[Z_m]\to 0$. Equivalently,
\(
\mathbb{E}\big[(\alpha_m-\alpha)^2\big]\to 0,
\)
i.e.\ $\alpha_m\to\alpha$ in $\mathcal L^2(\Sigma, \mathcal U)$. Therefore, we obtain
\begin{equation}\label{eqn:alpha_m to alpha}
   \mathbb{E}\big[(\alpha_m-\alpha)^2\|\mathcal A\|^2\big] = \|\mathcal A\|^2\,\mathbb{E}[(\alpha_m-\alpha)^2] \to 0. 
\end{equation}
In the case \(\|\mathcal{A}\|=0\), the definition of \(\alpha=\alpha_k\) in Algorithm~\ref{alg:buildtree_exact} yields \(\alpha_k=\zeta_1\). An argument identical to that used above shows that
\(
\alpha_m \to \zeta_1.
\)
Therefore, the convergence asserted in~\eqref{eqn:alpha_m to alpha} also holds in this case.
Now, using \eqref{eqn:A_m to A} and \eqref{eqn:alpha_m to alpha} in \eqref{eq:decomp}, we obtain
\begin{equation}\label{eqn:II}
    \text{II} \coloneq \mathbb{E}\!\left[\|\alpha_m  \mathcal A_m - \alpha  \mathcal A\|^2\right] \to 0
\quad \text{as } m\to\infty.
\end{equation}
Similarly, we consider the term \(\mathrm{III}\). Without loss of generality, suppose that
\(
\|\Delta_{u_k}u_k\|\neq 0,
\)
as the case \(\|\Delta_{u_k}u_k\|=0\) follows immediately. Then, by the definition of \(\beta_k^{(m)}\) in~\eqref{beta}, one obtains
\(
\beta_k^{(m)} \to \beta_k
\) in $\mathcal L^2(\Sigma, \mathcal U)$,
in complete analogy with the convergence $\alpha_k^{(m)} 
\to \alpha_k$ in $\mathcal L^2(\Sigma, \mathcal U)$.
As a consequence, we obtain
\begin{align}\label{eq:III}
\text{III} 
&\le 2\,\mathbb{E}\Big[\|\beta_k^{(m)} \Delta_{u_k^{(m)}} (u_k^{(m)} - u_k)\|^2\Big] 
   + 2\,\mathbb{E}\Big[\|(\beta_k^{(m)} \Delta_{u_k^{(m)}} - \beta_k \Delta_{u_k}) u_k\|^2\Big] \nonumber\\
&\le 2\nu_2^2\, \mathbb{E}\Big[\|\Delta_{u_k^{(m)}} (u_k^{(m)} - u_k)\|^2 \Big] 
   + 4\,\mathbb{E}\Big[\| (\beta_k^{(m)} - \beta_k) \Delta_{u_k^{(m)}} u_k\|^2\Big] \nonumber\\
&\qquad + 4\nu_2^2\, \mathbb{E}\Big[\| (\Delta_{u_k^{(m)}} - \Delta_{u_k}) u_k \|^2\Big] \to 0 \quad \text{as } m \to \infty.
\end{align}
Here, the last convergence follows from the facts that \(\mathbb{E}[|\beta_k^{(m)} - \beta_k|^2] \to 0\), the induction hypothesis, and Lemma~\ref{lemma: contraction of Delta}.
Combining \eqref{eqn:II}, \eqref{eq:III} with \eqref{eq:exp_bound} and the induction hypothesis, 
we conclude that  
\(
\mathbb{E}\!\left[\|u_{k+1}^{(m)} - u_{k+1}\|^2\right] \to 0 \) as \( m \to \infty.
\)  
This completes the induction step, and therefore the proof.
\end{proof}

\subsection{Convergence for noisy data}\label{subsec: nonnoisy} In this subsection, we  establish the regularization property of the \texttt{E-IRMGL+$\Psi$} method introduced in Algorithm~\ref{alg:buildtree}. Specifically, we provide a detailed convergence analysis in the theorem stated below. Apart from established properties, the proof depends on several other results which are provided in Appendix~\ref{sec:appendix}.
\begin{theorem}\label{thm: non noisy}
Assuming all the convictions of Lemma~\ref{lemma: stability} are true and $k_m$ is the stopping index  for \Cref{alg:buildtree} as specified in \cref{rule:statistical_discrepancy} and let $K_{\max}(m)$ satisfy \eqref{eq:budget}. Then a
solution $u^\dagger$ of \eqref{Model eqn} exists such that, for every
$p\in(0,2)$,
\[
  \lim_{m\to\infty}\mathbb E\Big[\big\|u^{(m)}_{k_m}-u^\dagger\big\|^{p}\Big]=0 .
\]
In particular, $u^{(m)}_{k_m}\to u^\dagger$ in probability. 
\end{theorem}

\begin{proof}
Let $u^\dagger$ be the solution of \eqref{Model eqn} extablished by
Theorem~\ref{Convergence for exact data}, so that $\|u_k-u^\dagger\|\to0$,
where $\{u_k\}_{k\ge0}$ are the iterates of
Algorithm~\ref{alg:buildtree_exact}. Fix $p\in(0,2)$, put
$c_p:=\max\{1,2^{\,p-1}\}$, and let $\varepsilon>0$.
Exactly one of the following holds for the exact-data sequence
$\{u_j\}_{j\ge0}$:
\begin{enumerate}
\item[(i)] $\F(u_j)\neq v$ for every $j\ge0$,
\item[(ii)] there is a smallest $K_0\in\mathbb N$ with $\F(u_{K_0})=v$. Then
  Proposition~\ref{prop:stationary} gives $u_j=u_{K_0}$ for all $j\ge K_0$ and
  $u_{K_0}=u^\dagger$.
\end{enumerate}
By $\|u_k-u^\dagger\|\to0$ choose 
$k=k(\varepsilon)\in\mathbb N$ with
\begin{equation}\label{eq:choice-k}
  \|u_k-u^\dagger\|^{p}\le\varepsilon,\qquad\text{and, in case (ii),
  additionally } k\ge K_0 .
\end{equation}
Split $\{0,1,\dots,k-1\}=\mathcal J_{\neq}\,\dot\cup\,\mathcal J_{=}$ with
$\mathcal J_{\neq}:=\{j<k:\F(u_j)\neq v\}$ and
$\mathcal J_{=}:=\{j<k:\F(u_j)=v\}$. By (ii), $u_j=u^\dagger$ for every
$j\in\mathcal J_{=}$, and $\mathcal J_{=}=\emptyset$ in case (i). Both index sets are finite and independent of $m$.
We proceed with the proof in following steps.

\medskip\noindent
\textbf{Step 1.}
Let $m$ be large enough that $K_{\max}(m)\ge k$. The events
$\Gamma_m^{c}$, $\Gamma_m\cap\{k_m\ge k\}$ and
$\Gamma_m\cap\{k_m=j\}$, $j=0,\dots,k-1$, are pairwise disjoint and exhaust
the sample space, and hence
\begin{equation}\label{eqn:expectation_split_1}
\begin{split}
     \mathbb E\Big[\|u^{(m)}_{k_m}-u^\dagger\|^{p}\Big]
  &=\underbrace{\mathbb E\Big[\|u^{(m)}_{k_m}-u^\dagger\|^{p}\chi_{\Gamma_m^{c}}\Big]}_{\rm I}
  +\underbrace{\mathbb E\Big[\|u^{(m)}_{k_m}-u^\dagger\|^{p}\chi_{\Gamma_m\cap\{k_m\ge k\}}\Big]}_{\rm II} \\
  &\qquad +\underbrace{\sum_{j=0}^{k-1}\mathbb E\Big[\|u^{(m)}_{j}-u^\dagger\|^{p}\chi_{\Gamma_m\cap\{k_m=j\}}\Big]}_{\rm III},
\end{split}
\end{equation}
where $\chi_{\Gamma_m}$ denotes the characteristic function of $\Gamma_m$. Note that, in $\rm III$ term the random index has already been replaced by the
deterministic index $j$, since $k_m=j$ on the corresponding event.

\medskip\noindent
\textbf{Step 2.}
Since $k_m\le K_{\max}(m)$ pointwise,
\[
  \big\|u^{(m)}_{k_m}-u^\dagger\big\|
  \;\le\;\max_{0\le j\le K_{\max}(m)}\big\|u^{(m)}_{j}-u^\dagger\big\|
  \qquad\text{a.s.},
\]
the right-hand side involving deterministic indices only, so that
Proposition~\ref{prop:apriori} applies and yields
$\mathbb E\big[\|u^{(m)}_{k_m}-u^\dagger\|^{2}\big]\le M^{2}$ with $M$
independent of $m$. Since moreover
$\mathbb P(\Gamma_m^{c}) \to0$ as $m \to \infty$, it
suffices to combine the two through H\"older's inequality with the conjugate
exponents $2/p$ and $2/(2-p)$ as
\[
  {\rm I}\;\le\;
  \Big(\mathbb E\Big[\|u^{(m)}_{k_m}-u^\dagger\|^{2}\Big]\Big)^{p/2}
  \,\mathbb P\big(\Gamma_m^{c}\big)^{\,1-p/2}
  \;\le\;M^{p}\,\mathbb P\big(\Gamma_m^{c}\big)^{\,1-p/2}\;\to\;0 \quad \text{as } m \to \infty.
\]

\medskip\noindent
\textbf{Step 3.}
On the event $\Gamma_m\cap\{k_m\ge k\}$, Lemma~\ref{Lemma: Monotonicity} applies to every
index $k\le j<k_m$ and yields, pointwise on that event,
$\|u^{(m)}_{k_m}-u^\dagger\|\le\|u^{(m)}_{k}-u^\dagger\|$. Consequently the
index under the expectation becomes deterministic and, using
$(a+b)^{p}\le c_p(a^{p}+b^{p})$, Jensen's inequality and \eqref{eq:choice-k} yields
\[
  {\rm II}\;\le\;\mathbb E\Big[\|u^{(m)}_{k}-u^\dagger\|^{p}\Big]
  \;\le\;c_p\,\mathbb E\Big[\|u^{(m)}_{k}-u_k\|^{p}\Big]+c_p\,\|u_k-u^\dagger\|^{p}
  \;\le\;c_p\Big(\mathbb E\big[\|u^{(m)}_{k}-u_k\|^{2}\big]\Big)^{p/2}+c_p\,\varepsilon.
\]
By Lemma~\ref{lemma: stability} at the fixed index $k$ the first summand tends to $0$,
so $\limsup_{m\to\infty}{\rm II}\le c_p\,\varepsilon$.

\medskip\noindent
\textbf{Step 4.}
The sum in $\rm III$ term consists of $k=k(\varepsilon)$ terms, all carrying deterministic
indices, and $k$ does not depend on $m$. For $j\in\mathcal J_{\neq}$,
H\"older's inequality with the exponents $2/p$ and $2/(2-p)$ together with
Proposition~\ref{prop:apriori} and Lemma~\ref{lem:noearlystop} give
\[
  \mathbb E\Big[\|u^{(m)}_{j}-u^\dagger\|^{p}\chi_{\{k_m=j\}}\Big]
  \;\le\;\Big(\mathbb E\big[\|u^{(m)}_{j}-u^\dagger\|^{2}\big]\Big)^{p/2}
        \mathbb P\big(k_m=j\big)^{1-p/2}
  \;\le\;M^{p}\,\mathbb P\big(k_m=j\big)^{1-p/2} \to 0
\]
as $m \to \infty$. For $j\in\mathcal J_{=}$ we have
$u_j=u^\dagger$, hence, by Jensen's inequality and Lemma~\ref{lemma: stability} at the fixed
index $j$,
\[
  \mathbb E\Big[\|u^{(m)}_{j}-u^\dagger\|^{p}\chi_{\{k_m=j\}}\Big]
  \le\mathbb E\Big[\|u^{(m)}_{j}-u_j\|^{p}\Big]
  \le\Big(\mathbb E\big[\|u^{(m)}_{j}-u_j\|^{2}\big]\Big)^{p/2}\to 0 \quad \text{as } m \to \infty.
\]
Being a finite sum of null sequences, ${\rm III}\to0$ as $m\to\infty$.

\medskip\noindent
\textbf{Step 5.} Combining Steps 2--4 in
\eqref{eqn:expectation_split_1} gives
\[
  \limsup_{m\to\infty}\mathbb E\Big[\|u^{(m)}_{k_m}-u^\dagger\|^{p}\Big]
  \;\le\;c_p\,\varepsilon .
\]
Since $\varepsilon>0$ was arbitrary, the limit equals $0$. Convergence in
probability follows from Markov's inequality.
\end{proof}

\section{Convergence analysis under heuristic rule}\label{sec:heuristic}
\noindent In this section, we analyze the regularization properties of \texttt{E-IRMGL+$\Psi$} \eqref{main iterative scheme} when terminated according to \cref{rule:heuristic_discrepancy}. Throughout this analysis, we assume that Assumption~\ref{assump:empirical} holds. 
While the theoretical justification of rule~\ref{rule:heuristic_discrepancy} is weaker because it relies on Assumption~\ref{assump:empirical}, it is included primarily to enable a numerical comparison between the two stopping rules for the proposed method.
To facilitate the subsequent discussion, we present the scheme in the form of \Cref{alg:heuristic}.
\begin{algorithm}
\caption{\texttt{E-IRMGL+\(\Psi\)} with heuristic principle}
\label{alg:heuristic}
\begin{algorithmic}[1]
\State \textbf{Input:} Operator $\F$, measurements $v_1, v_2, \ldots v_m,$ initial reconstructor $\Psi_{\theta_m}$, parameters $\theta, \lambda, R,$ and $\varrho \geq 1$.
\State \textbf{Set:} $\hat v^{(m)} \coloneq \frac{1}{m} \sum_{i=1}^{m}v_i$ and $z_m \coloneq \sqrt{\frac{1}{m-1} \sum_{i=1}^{m} \|v_i - \hat v^{(m)}\|^2}.$
\State \textbf{Initialize:} $u_0^{(m)} \coloneq \Psi_{         \theta_m}(\hat{v}^{(m)})$.
\State \textbf{Iteration:} For $k = 0,1,\dots,k_\infty$ (with $u_k^{(m)} \in \mathcal{D}(\F)$), compute
    \begin{equation*}
       u_{k+1}^{(m)} \coloneq u_k^{(m)} 
       - \alpha_k^{(m)} \F'(u_k^{(m)})^{*}\!\left(\F(u_k^{(m)}) - \hat v^{(m)}\right)
       - \beta_k^{(m)} \Delta_{u_k^{(m)}}u_k^{(m)},
    \end{equation*}
    where $\alpha_k^{(m)}$ and $\beta_k^{(m)}$ are given by~\eqref{alpha} and~\eqref{beta}, and 
    $\Delta_{u_k^{(m)}}$ is defined in \Cref{subsec: Graph Theory}.
\State \textbf{Stopping rule:} Determine the index
    \[
      k_m^* \in \arg\min \Bigl\{ \Omega(k,\hat v^{(m)}) \;:\; 0 \leq k \leq k_\infty \Bigr\},
      \quad \text{where }\]
   \[  \Omega(k, \hat v^{(m)}) \coloneq (k+\varrho)\|\F(u_k^{(m)}) - \hat v^{(m)}\|^2.
    \]
\State \textbf{Output:} $u_{k_m^*}^{(m)}$.
\end{algorithmic}
\end{algorithm} 
\begin{remark}
    Rule~\ref{rule:heuristic_discrepancy} is straightforward to implement. During iteration, we record 
$\Omega(k, \hat v^{(m)})$ against $k$ and, after sufficient iterations, select 
$k^*_m$ that minimizes it. Since the iterative regularization methods for nonlinear problems may exhibits 
only local convergence~\cite{hubmer2022numerical}, so the maximum iteration count must be 
chosen large enough to avoid mistaking a local for the global minimum. 
The rule requires a fixed constant $\varrho$, introduced in~\cite{jin2017heuristic} for the 
augmented Lagrangian method, where a similar stopping strategy was used. 
Choosing $\varrho$ sufficiently large ensures an accurate solution and prevents 
$k_m^*$ from being underestimated. Previous numerical and analytical studies
\cite{jin2017heuristic,real2024hanke, zhang2018heuristic} indicate that the lower bound of $\varrho$ depends on the 
regularization method and for the Landweber iteration $\varrho \ge 1$ is essential. In our method we also require that $\varrho \geq 1$,
which can be seen later in Lemma~\ref{lemmma:heuristic_first}.
\end{remark}
\subsection{Preliminary analysis}
We first establish that  
\begin{equation}\label{eq:Omega_to_zero}
    \Omega(k_m^*, \hat v^{(m)}) \to 0 \quad \text{as } m \to \infty,
\end{equation}  
where \(\Omega(\cdot, \cdot)\) is defined in \cref{rule:heuristic_discrepancy}.  
To this end, we introduce an auxiliary index \(\hat{k}_m\), determined by the stopping rule as the smallest integer satisfying  
\begin{equation}\label{eq:stopping_rule_modifed}
    \left\|\mathcal{F}(u^{(m)}_{\hat k_m}) - \hat v^{(m)}\right\|^2 
    + \frac{\mathcal{M}}{(\hat k_m + \varrho)^2} 
    \leq \frac{\tau_m^2}{m}\, z_m^2,
\end{equation}  
where \(\tau_m \geq 1\) may depend on \(m\), and  
\(
    \mathcal{M} = \frac{3\mathcal{H}\wp^2}{2(1+\eta) \zeta_1}.
\)
It is worth noting that the stopping rule in \eqref{eq:stopping_rule_modifed} is a modification of 
\cref{rule:statistical_discrepancy} through the inclusion of the additional term 
$\tfrac{\mathcal M}{(k + \varrho)^2}$. The presence of this term ensures that 
$\hat k_m \to \infty$ as $m \to \infty$, which is a crucial ingredient in the proof of 
\eqref{eq:Omega_to_zero}.

\begin{lemma}\label{lemmma:heuristic_first}
Suppose that $\mathcal{C}>0$ as defined in \eqref{assum: On C} and all the assumptions of Lemma~\ref{Lemma: Monotonicity} satisfies, then the following hold on the event $\Gamma_m$. 
\begin{enumerate}
\item[$(i.)$] $u_k^{(m)} \in \mathcal{B}_{3\wp} (u_0)$ for all $0 \leq k \leq \hat k_m,$ where $\hat k_m$ is the stopping index as in \eqref{eq:stopping_rule_modifed}.
    \item[$(ii.)$] $\hat k_m$ 
is a finite integer and $\hat k_m \to \infty$ as $m \to \infty.$
\end{enumerate}
\end{lemma}
\begin{proof}
    Following the same procedure of Lemma~\ref{Lemma: Monotonicity}, we have 
    \begin{equation}\label{split_eqn_heuristic}
    \|g_{k+1}^{(m)}\|^2 - \|g_k^{(m)}\|^2 = \|h_k^{(m)}\|^2 - 2\langle g_k^{(m)}, h_k^{(m)} \rangle
\end{equation}
along with the estimates 
\begin{align}\label{eqn:xk_dk}
  - \langle g_k^{(m)}, h_k^{(m)} \rangle & \leq - (1 - \eta)\alpha_k^{(m)}\|\F(u_k^{(m)}) - \hat v^{(m)}\|^2 + \wp\nu_0 \|\F(u_k^{(m)}) -\hat  v^{(m)}\|^2 \\ 
   \nonumber 
   & \quad + (1 + \eta)\alpha_k^{(m)}\|v - \hat v^{(m)}\|\|\F(u_k^{(m)}) - \hat v^{(m)}\|\\ \label{norm_dk}
   \|h_k^{(m)}\|^2  &\leq 2(\zeta_0^2 + \nu_0 \nu_1)\|\F(u_k^{(m)}) - \hat{v}^{(m)}\|^2. 
\end{align}
Now, utilizing the property of the space $\Gamma_m$ and according to the definition of $\hat k_m,$ for $k < \hat k_m$ we have
\[\|v - \hat v^{(m)}\| \leq \frac{\tau_m}{\mathcal{H}\sqrt{m}}z_m \leq \frac{1}{\mathcal{H}}\left( \|\F(u_k^{(m)}) - \hat{v}^{(m)}\|^2 + \frac{\mathcal{M}}{(k + \varrho)^2} \right)^{1/2}\]
and by the inequality $2a_1 a_2 \leq a_1^2 + a_2^2$ for $a_1, a_2 \in \mathbb{R}$, we further have
\begin{equation}\label{eqn:residual_and_km}
 \|v - \hat v^{(m)}\|\|\F(u_k^{(m)}) - \hat v^{(m)}\| 
 \leq \frac{1}{\mathcal{H}} \left( \|\F(u_k^{(m)}) - \hat v^{(m)}\|^2 + \frac{\mathcal M}{2(k + \varrho)^2} \right).
\end{equation}
Using \eqref{eqn:xk_dk}, \eqref{norm_dk} and \eqref{eqn:residual_and_km} in \eqref{split_eqn_heuristic},  we get
\begin{align*}
    \|g_{k+1}^{(m)}\|^2 - \|g_k^{(m)}\|^2 \leq -2\mathcal C \|\F(u_k^{(m)}) - \hat v^{(m)}\|^2 + \frac{ \mathcal  M(1+\eta)\zeta_1}{\mathcal{H}( k + \varrho)^2}.
\end{align*}
Now, taking the sum from $n=0$ to $n=k$ on both side, we further obtain
\begin{equation}\label{eqn:heuristic last}
     \|g_{k+1}^{(m)}\|^2 + 2 \mathcal C \sum_{n=0}^k \|\F(u_n^{(m)}) - \hat v^{(m)}\|^2 \leq \|g_0^{(m)}\|^2 + \frac{ \mathcal M(1 + \eta)\zeta_1}{\mathcal{H}} \sum_{n=0}^k \frac{1}{(n+ \varrho)^2}.
\end{equation}
Note that $\varrho \geq 1, \mathcal{M} = \frac{3\mathcal{H}\wp^2}{2(1 + \eta) \zeta_1}$ and since $ \sum_{n=0}^{k} \frac{1}{(n+\varrho)^2} \leq \sum_{n=0}^{\infty} \frac{1}{(n+\varrho)^2} = \frac{\pi^2}{6} <2,$ it yields
\[\|u_{k+1}^{(m)} - u^\dagger\|^2 \leq \|u_0^{(m)} - u^\dagger\|^2 + 3\wp^2 \leq 4\wp^2.\]
Hence, we have \(u_{k+1}^{(m)} \in \mathcal{B}_{2\wp}(u^\dagger).\) Since $\|u_0 - u^\dagger\| \leq \wp,$ we therefore have $u_{k+1}^{(m)} \in \mathcal{B}_{3\wp}(u_0)$.
To prove $(ii)$, let us suppose a contrary, that is $\hat{k}_m$ is infinite. 
Then from \eqref{eqn:heuristic last}, we have
\begin{equation}\label{eqn:heuristic_residual_finite}
     2\mathcal C \sum_{n=0}^k \|\F(u_n^{(m)}) - \hat v^{(m)}\|^2 \leq 4\wp^2 \quad \forall \; k \geq 0.
\end{equation}
Additionally, the definition of $\hat k_m$ tells us that for every integer $k \geq 0$ we have 
\[ \frac{\tau_m^2}{m}z_m^2 - \frac{\mathcal M}{(k + \varrho)^2} <\|\F(u_k^{(m)}) - \hat v^{(m)}\|^2.\]
Summing up the above inequality from $n=0$ to $n=k < \hat k_m$ and then using \eqref{eqn:heuristic_residual_finite},
\[ \frac{ \tau_m^2z_m^2}{m}(k+1) \leq \mathcal{M} \sum_{n=0}^k\frac{1}{(n + \varrho)^2} + \frac{2\wp^2}{\mathcal C} \leq \frac{\mathcal{M}\pi^2}{6} +  \frac{2\wp^2}{\mathcal C}.  \]
Using the estimate $z_m \geq \sigma/2,$ which holds on $\Gamma_m,$ we obtain
\[\frac{\tau_m^2\sigma^2}{4m}(k+1) \leq \frac{\mathcal{M}\pi^2}{6} +  \frac{2\wp^2}{\mathcal C}< \infty.\]
Letting $k \to \infty,$ we get a contradiction. Hence $\hat k_m$ must be finite. Finally from \eqref{eq:stopping_rule_modifed},
\[ \frac{\mathcal M}{(\hat k_m + \varrho)^2} 
    \leq \left(\frac{\tau_m}{\sqrt{m}} \right)^2 z_m^2 \to 0 \quad \text{as } m \to \infty. \]
    Therefore, we must have $\hat k_m \to \infty$ as $m \to \infty.$ Thus the result.
\end{proof}
By virtue of Assumption~\ref{assump:empirical}, it directly follows that 
\cref{rule:heuristic_discrepancy} determines a finite index $k_m^*$. 
If $k_\infty$ is already finite, the statement is evident. Thus, we restrict 
our attention to the case $k_\infty = \infty$. According to 
\cref{rule:heuristic_discrepancy}, we have
\begin{equation}\label{eqn:above_Lemma_4.3}
 \Omega(k, \hat v^{(m)}) = (k + \wp)
\| \F(u_k^{(m)}) - \hat v^{(m)} \|^2
\ge (k + \wp)\kappa^2
\|v - \hat v^{(m)}\|^2 \to \infty
\end{equation}
as $k \to \infty$. Hence, a finite integer $k_m^*$ must exist at which 
$\Omega(k, \hat v^{(m)})$ attains its minimum. Building upon this observation, the next lemma establishes the main result of this subsection, and its proof is motivated by the approaches outlined in \cite{zhang2018heuristic} and \cite{real2024hanke}.
\begin{lemma}\label{lemma:heuristic_secound}
Suppose that $\mathcal{C} > 0$, as specified in \eqref{assum: On C}, and that all the assumptions of Lemma~\ref{Lemma: Monotonicity} are satisfied. Let $k_m^*$ denote the stopping index determined according to rule~\ref{rule:heuristic_discrepancy}.
 Then, on the event $\Gamma_m$, 
    \[\Omega(k_m^*, \hat v^{(m)}) \to 0 \quad \text{as } m \to \infty.\]
Consequently, 
\[\|\mathcal{F}(u_{k_m^*}^{(m)}) - \hat v^{(m)}\| \to 0 \quad \text{and} \quad k_m^* \|\hat v^{(m)} - v\|^2 \to 0\; \text{ as } m \to \infty.\]
\end{lemma}
\begin{proof}
   By virtue of \eqref{eqn:heuristic_residual_finite} established in the proof of Lemma~\ref{lemmma:heuristic_first}, we have
    \[ \sum_{k=0}^{\hat k_m -1} \|\F(u_k^{(m)}) - \hat v^{(m)}\|^2 \leq \frac{2\wp^2}{\mathcal C}.\]
    By the minimality of $\Omega(k_m^*, \hat v^{(m)})$, we have
\[
\|\F(u_k^{(m)}) - \hat v^{(m)}\|^2 = \frac{\Omega(k, \hat v^{(m)})}{k+\varrho} 
   \geq \frac{\Omega(k_m^*, \hat v^{(m)})}{k + \varrho}.
\]
Summing over $k = 0,1,\dots, \hat k_m-1$, it follows that
\[
\sum_{k=0}^{\hat k_m-1} \frac{1}{k+\varrho}\,\Omega(k_m^*, \hat v^{(m)}) 
   \leq \sum_{k=0}^{\hat k_m-1} \|\F(u_k^{(m)}) - \hat v^{(m)}\|^2
   \leq \frac{2\wp^2}{\mathcal C}.
\]
Moreover,
\[
\sum_{k=0}^{\hat k_m-1} \frac{1}{k+\varrho} 
   \geq \int_0^{\hat k_m} \frac{dt}{t+\varrho} 
   = \log\!\left(\frac{\hat k_m+\varrho}{\varrho}\right).
\]
Hence,
\[
\Omega(k_m^*, \hat v^{(m)}) 
   \leq \frac{2 \wp^2}{\mathcal C \, \log\!\big(\tfrac{\hat k_m+\varrho}{\varrho}\big)}.
\]
Since $\hat k_m \to \infty$ as $m \to \infty$, we deduce that 
\(
    \Omega(k_m^*, \hat v^{(m)}) \to 0.
\)
The remaining assertions then follow directly from this convergence, together with Assumption~\ref{assump:empirical} and \eqref{eqn:above_Lemma_4.3}.
\end{proof}
\subsection{Convergence}
We now turn to the proof of the main result of this section. Prior to that, we establish two lemmas that will play a key role in its derivation.
\begin{lemma}\label{eqn:lemma_third}
Suppose that $\mathcal{C} > 0$, as specified in \eqref{assum: On C}, and that all the assumptions of Lemma~\ref{Lemma: Monotonicity} are satisfied. Let $k_m^*$ denote the stopping index determined according to \cref{rule:heuristic_discrepancy}. Then, on the event $\Gamma_m$, $u_{k}^{(m)} \in \mathcal{B}_{3\wp}(u_0)$ for all $0 \leq k \leq k_m^*$, if $m$ is sufficiently large.
\end{lemma}
\begin{proof}
    We proceed by induction. The claim is immediate for $k = 0$. 
Assume that $u_k^{(m)} \in \mathcal{B}_{3\wp}(u_0)$  for some $k < k_m^*$. Using \eqref{split_eqn_heuristic}, \eqref{eqn:xk_dk} and \eqref{norm_dk} we will show that $u_{k+1}^{(m)} \in \mathcal{B}_{3\wp}(u_0)$. By virtue of Young’s inequality, it follows that for any $\mathcal H> 0,$
\[ \|v - \hat v^{(m)}\| \|\F(u_k^{(m)}) - \hat v^{(m)}\| \leq \frac{1}{\mathcal H} \|\F(u_k^{(m)}) - \hat v^{(m)}\|^2 + \frac{\mathcal H}{4}  \|v - \hat v^{(m)}\|^2. \]
Thus \eqref{eqn:xk_dk} becomes 
\begin{align}\label{norm_dk_heuristic}
      - \langle g_k^{(m)}, h_k^{(m)} \rangle 
      & \leq - (1 - \eta)\alpha_k^{(m)}\|\F(u_k^{(m)}) - \hat v^{(m)}\|^2 + \wp\nu_0 \|\F(u_k^{(m)}) -\hat  v^{(m)}\|^2 \\ 
   \nonumber 
   & \quad + (1 + \eta)\alpha_k^{(m)}\left (\frac{\mathcal H}{4}\|v - \hat v^{(m)}\|^2 + \frac{1}{\mathcal H} \|\F(u_k^{(m)}) - \hat v^{(m)}\|^2 \right).
\end{align}
Combining \eqref{split_eqn_heuristic}, \eqref{norm_dk_heuristic} and \eqref{norm_dk} we have
\begin{align*}
\|g_{k+1}^{(m)}\|^2 &- \|g_k^{(m)}\|^2 \leq \frac{(1+\eta)\zeta_1 \mathcal H}{2} \|v - \hat{v}^{(m)}\|^2 \\
&  -2 \left[ \left( 1-\eta  - \frac{1 + \eta}{\mathcal H}\right)\zeta - \nu_0 (\wp + \nu_1) - \zeta_0^2 \right] \|\F(u_k^{(m)}) -\hat  v^{(m)}\|^2 \\
& \leq - 2 \mathcal C \|\F(u_k^{(m)}) -\hat  v^{(m)}\|^2 + \mathcal C_1 \|v - \hat{v}^{(m)}\|^2
\end{align*}
where $ \mathcal C_1 =\frac{(1+\eta)\zeta_1 \mathcal H}{2}$ and $ \mathcal C$ is the same positive constant as in \eqref{assum: On C}. Therefore on $\Gamma_m$, we have
\begin{align*}
    \|g_{k+1}^{(m)}\|^2 + 2 \mathcal C \sum_{n=0}^{k}\|\F(u_n^{(m)}) -\hat  v^{(m)}\|^2  & \leq  \|g_0\|^2 +  \mathcal C_1 (k+1)\|v -\hat v^{(m)}\| \\
    & \leq \wp^2 + \mathcal{C}_1 k_m^* \|v - \hat v^{(m)}\|^2.
\end{align*}
By Lemma~\ref{lemma:heuristic_secound} we can guarantee that
\(\mathcal{C}_1 k_m^* \|v - \hat v^{(m)}\|^2 \leq 3\wp^2\)
 for sufficiently large $m$. Thus, we have
\(\|u_{k+1}^{(m)} - u^\dagger\|^2 \leq 4\wp^2.\)
 Since \(\|u_0 - u^\dagger\| \leq \wp,\) we therefore have $u_{k+1}^{(m)} \in \mathcal{B}_{3\wp}(u_0)$.   
\end{proof}
\begin{lemma}\label{lem:comparison}
Under the assumptions of Theorem~\ref{Convergence for exact data},
\[
  \liminf_{k\to\infty}\;(k+\varrho)\big\|\F(u_k)-v\big\|^{2}\;=\;0.
\]
Consequently, for every $\epsilon>0$ there exist arbitrarily large indices $k$ with
$$(k+\varrho)\|\F(u_k)-v\|^{2}\le\epsilon \quad \text{and} \quad \|u_k-u^\dagger\|^{2}\le\epsilon.$$
\end{lemma}
\begin{proof}
 By \eqref{sum for non noisy case},
$\sum_{k=0}^{\infty} \|\F(u_k)-v\|^{2}<\infty$. If one had $\liminf_k (k+\varrho)\|\F(u_k)-v\|^{2}=2c>0$, then
$\|\F(u_k)-v\|^{2}\ge c/(k+\varrho)$ for all large $k$ and hence $\sum_{k=0}^{\infty} \|\F(u_k)-v\|^{2}=\infty$, a
contradiction. The last assertion is obtained by coupling this result with Theorem~\ref{Convergence for exact data}, which guarantees that $\|u_k - u^\dagger\| \to 0$ as $k \to \infty$.
\end{proof}

\noindent We are now in a position to present the main result of this section concerning \Cref{alg:heuristic}. The proof relies on the preceding lemma together with the stability result stated in Lemma~\ref{lemma: stability} and \Cref{Convergence for exact data}.




\begin{theorem}\label{thm:heuristic}
Suppose that $\mathcal C>0$, as specified in \eqref{assum: On C}, and that all the
assumptions of Lemma~\ref{Lemma: Monotonicity} are satisfied. Let $k^*_m$ denote the stopping index
determined according to \cref{rule:heuristic_discrepancy} for
\Cref{alg:heuristic}. Then a solution
$u^\dagger$ of \eqref{Model eqn} exists such that, for every $p\in(0,2)$,
\[
  \lim_{m\to\infty}\mathbb E\Big[\big\|u^{(m)}_{k^*_m}-u^\dagger\big\|^{p}\Big]=0 .
\]
In particular, $u^{(m)}_{k^*_m}\to u^\dagger$ in probability.
\end{theorem}

\begin{proof}
Let $u^\dagger$ be the solution of \eqref{Model eqn} established by
Theorem~\ref{Convergence for exact data}, so that $\|u_k-u^\dagger\|\to0$,
where $\{u_k\}_{k\ge0}$ are the iterates of the exact-data version of
\Cref{alg:heuristic}. Fix $p\in(0,2)$, set $c_p:=\max\{1,2^{\,p-1}\}$,
and let $\varepsilon>0$.
Exactly one of the following holds for the exact-data sequence:
\begin{enumerate}
    \item [(i)] $\F(u_j)\neq v$ for every $j\ge0$, 
    \item [(ii)] there is a smallest $K_0\in \mathbb{N}$ with
$\F(u_{K_0})=v$, in which case Proposition~\ref{prop:stationary} gives
$u_j=u_{K_0}=u^\dagger$ for all $j\ge K_0$.
\end{enumerate}
By Lemma~\ref{lem:comparison}, we choose $k = k(\varepsilon)$ such that
\begin{equation}\label{eq:choice_j}
  \|u_k - u^\dagger\|^2 \le \varepsilon
  \qquad\text{and}\qquad
  (k + \varrho) \|\mathcal{F}(u_k) - v\|^2 \le \varepsilon,
\end{equation}
enforcing $k \ge K_0$ in case (ii). We partition the index set $\{0, \dots, k-1\} = \mathcal{J}_{\neq} \mathbin{\dot{\cup}} \mathcal{J}_{=}$ in the same manner as in Theorem~\ref{thm: non noisy}. The proof then follows the same steps as Theorem~\ref{thm: non noisy}, except for Step~3, which is given below.


\medskip\noindent
\textbf{Step 3.} Although the heuristic rule does not yield monotonicity, it guarantees the controlled growth estimate
\begin{equation*}
  \|g_{k+1}^{(m)}\|^2 - \|g_k^{(m)}\|^2 \le \mathcal{C}_1 \|v - \hat{v}^{(m)}\|^2,
\end{equation*}
for all $0 \le k < k_m^*$, as established in the proof of Lemma~\ref{eqn:lemma_third}. Summing this inequality over $k, \dots, k_m^* - 1$ yields, pointwise on the event $\Gamma_m \cap \{k_m^* \ge k\}$,
\begin{equation}\label{eq:quasi_sum}
  \big\|u^{(m)}_{k_m^*}-u^\dagger\big\|^{2}
  \le \big\|u^{(m)}_{k}-u^\dagger\big\|^{2}
  + \mathcal{C}_1 k_m^* \big\|v-\hat{v}^{(m)}\big\|^{2}.
\end{equation}
 By Assumption~\ref{assump:empirical} and then by the minimality of $k^*_m$ in rule~\ref{rule:heuristic_discrepancy}, applied with the comparison index
$k$, one has
\[
  k^*_m\big\|v-\hat v^{(m)}\big\|^{2}
  \;\le\;\frac{1}{\kappa^{2}}\,
        k^*_m\big\|\F(u^{(m)}_{k^*_m})-\hat v^{(m)}\big\|^{2}
  \;\le\;\frac{1}{\kappa^{2}}\,\Omega\big(k^*_m,\hat v^{(m)}\big)
  \;\le\;\frac{1}{\kappa^{2}}\,(k+\varrho)
        \big\|\F(u^{(m)}_{k})-\hat v^{(m)}\big\|^{2},
\]
Combining this
with \eqref{eq:quasi_sum} and using
$$\|\F(u^{(m)}_k)-\hat v^{(m)}\|\le\|\F(u^{(m)}_k)-\F(u_k)\|
 +\|\F(u_k)-v\|+\|\hat v^{(m)}-v\|,$$ 
 we obtain
 \begin{align*}
   {\rm II}&=  \mathbb{E}\left[\|u_{k_m^*}^{(m)} - u^\dagger\|^p \chi_{\Gamma_m \cap \{k_m^* \geq k\}} \right] \leq \mathbb{E}\left[\|u_{k_m^*}^{(m)} - u^\dagger\|^p \right] \leq  \mathbb{E}\Big[\big\|u^{(m)}_{k}-u^\dagger\big\|^{p}\Big] \\
     & +\Big(\tfrac{\mathcal C_1}{\kappa^{2}}\Big)^{p/2}(k+\varrho)^{p/2}\,
   c_{p,3}\Big\{\mathbb E\left[\big\|\F(u^{(m)}_k)-\F(u_k)\big\|^{p}\right]
          +\big\|\F(u_k)-v\big\|^{p}
          +\mathbb E\left[\big\|\hat v^{(m)}-v\big\|^{p}\right]\Big\},
 \end{align*}
where $c_{p,3}= \max\{1, 3^{p-1}\}$. We estimate the four contributions sequentially. First, we have the bound
\[
  \mathbb{E}\left[\|u^{(m)}_{k}-u^\dagger\|^{p}\right]
  \le c_p \left(\mathbb{E}\left[\|u^{(m)}_k-u_k\|^{2}\right]\right)^{p/2} + c_p \|u_k-u^\dagger\|^{p},
\]
where the first summand vanishes as $m \to \infty$ by Lemma~\ref{lemma: stability} at the fixed index $k$, and the second is bounded by $c_p \varepsilon^{p/2}$ according to \eqref{eq:choice_j}. Second term goes to zero as $m \to \infty$, by applying Lemma~\ref{lemma: stability} at the fixed index $k$ along with the continuity of $\mathcal{F}$.
Third, and most decisively, the choice of $k$ in \eqref{eq:choice_j} absorbs the factor $(k+\varrho)^{p/2}$ via
\[
  (k+\varrho)^{p/2} \|\mathcal{F}(u_k)-v\|^{p}
  = \left((k+\varrho) \|\mathcal{F}(u_k)-v\|^{2}\right)^{p/2}
  \le \varepsilon^{p/2}.
\]
Fourth, for fixed $k$, we obtain
\[
  (k+\varrho)^{p/2} \mathbb{E}\left[\|\hat{v}^{(m)}-v\|^{p}\right]
  \le (k+\varrho)^{p/2} \left(\frac{\sigma^{2}}{m}\right)^{p/2} \to 0 \quad \text{as } m \to \infty.
\]
Collecting these four bounds, we conclude that $\limsup_{m\to\infty} \mathrm{II} \le \left(c_p + (\mathcal C_1/\kappa^2)^{p/2} c_{p,3}\right) \varepsilon^{p/2}$.
The subsequent steps of the proof follows analogously to the similar case in the proof of \Cref{thm: non noisy}.  
\end{proof}
 \section{Numerical results and performance analysis}\label{sec: numerical}
This section details a series of numerical experiments designed to evaluate the effectiveness of the proposed algorithms, $\text{\Cref{alg:buildtree}}$ and $\text{\Cref{alg:heuristic}}$. Our analysis is particularly directed toward applications involving both linear and nonlinear ill-posed tomography problems.
To implement the proposed method effectively, selecting appropriate initial reconstructors $\Psi_\theta$ is critical. In our experiments, we adopt three commonly used reconstruction methods: FBP, Tikhonov, and TV, whose detailed descriptions can be found in \cite[Section 4.2]{bajpai2025convergence}. The discrepancy principle and generalized cross-validation are used to choose the regularization parameters for the TV and Tikhonov methods, respectively. The Hann filter is employed to perform the FBP reconstruction.

To simulate realistic measurement conditions, independent and identically distributed (i.i.d.) Gaussian noise is added to the exact data \( v \). Specifically, the \( i \)-th noisy measurement \( v_i \) is modeled as  
\[
v_i = v + \epsilon_i \frac{\xi}{\|\xi\|},
\]
where \( \epsilon_i \) denotes an independent Gaussian noise scalar, and \( \xi \) is a random vector drawn from the standard normal distribution.

\noindent The quantitative evaluation of the reconstructed images is performed using three standard image quality metrics: the relative  error (RRE), the \textit{Peak Signal-to-Noise Ratio (PSNR)}, and the \textit{Structural Similarity Index (SSIM)}~\cite{wang2004image}. Since the pixel intensities are normalized within the range \([0, 1]\), the PSNR is defined as  
\[
\text{PSNR} := 20 \log_{10} \left( \frac{1}{\|u^\dagger - u_{k_m}^{(m)}\|} \right),
\]
while the relative error is given by  
\[
\text{RRE} := \frac{\|u_{k_m}^{(m)} - u^\dagger\|}{\|u^\dagger\|},
\]
where, \(u_{k_m}^{(m)}\) denotes the reconstructed image obtained after \(k_m\) iterations of \Cref{alg:buildtree}, while  \(u^\dagger\) represents the ground truth image. When employing \Cref{alg:heuristic}, we instead denote the reconstructed image by \(u_{k_m^*}^{(m)}\) to reflect the use of the heuristic stopping criterion.

\subsection{X-ray Computed Tomography}
Computed Tomography (CT) is a vital medical imaging modality, used for diagnosing internal conditions like tumors, injuries, and fractures. The system operates by rotating an X-ray source around the subject, capturing the differential attenuation of beams across tissues via a detector array. The aggregated angular measurements constitute a sinogram. Mathematically, this process is modeled as a discretized linear system
\begin{equation}\label{CT1}
    \F_cu = \tilde v,
\end{equation}
where the matrix $\F_c \in \mathbb{R}^{q \times n}$ represents the discretized Radon transform, $u \in \mathbb{R}^n$ is the  two-dimensional  image vector (unknown), and $\tilde v \in \mathbb{R}^q$ is the measured, noisy sinogram. For a detailed treatment, see \cite{natterer2001mathematics}. The vectorization of an \(E \times F\) two-dimensional image, whose pixels are indexed by \(a = (i,j)\), is performed lexicographically using either a row-major or column-major ordering. This process maps the image to a vector in \(\mathbb{R}^{n}\), where  \(n = E \cdot F\).

\noindent In our tests, we incorporate the Operator Discretization Library (ODL) \cite{jonasadler20181442734} to describe the image domain. Specifically, the function \texttt{odl.uniform\_discr} is used to create a uniformly discretized $E \times F$ square grid over the region $[-128, 128]^2 \subset \mathbb{R}^2$.       
The Shepp--Logan phantom from \texttt{skimage.data} serves as the test image for the numerical simulations. The true image and the corresponding noisy sinograms, corrupted with Gaussian noise 
\(\epsilon_i \sim \mathcal{N}(0, \varepsilon^2)\) where \(\varepsilon = 10\), are presented in Fig.~\ref{fig:CT_true}.
\begin{figure}
    \centering
    \includegraphics[width=0.98\linewidth]{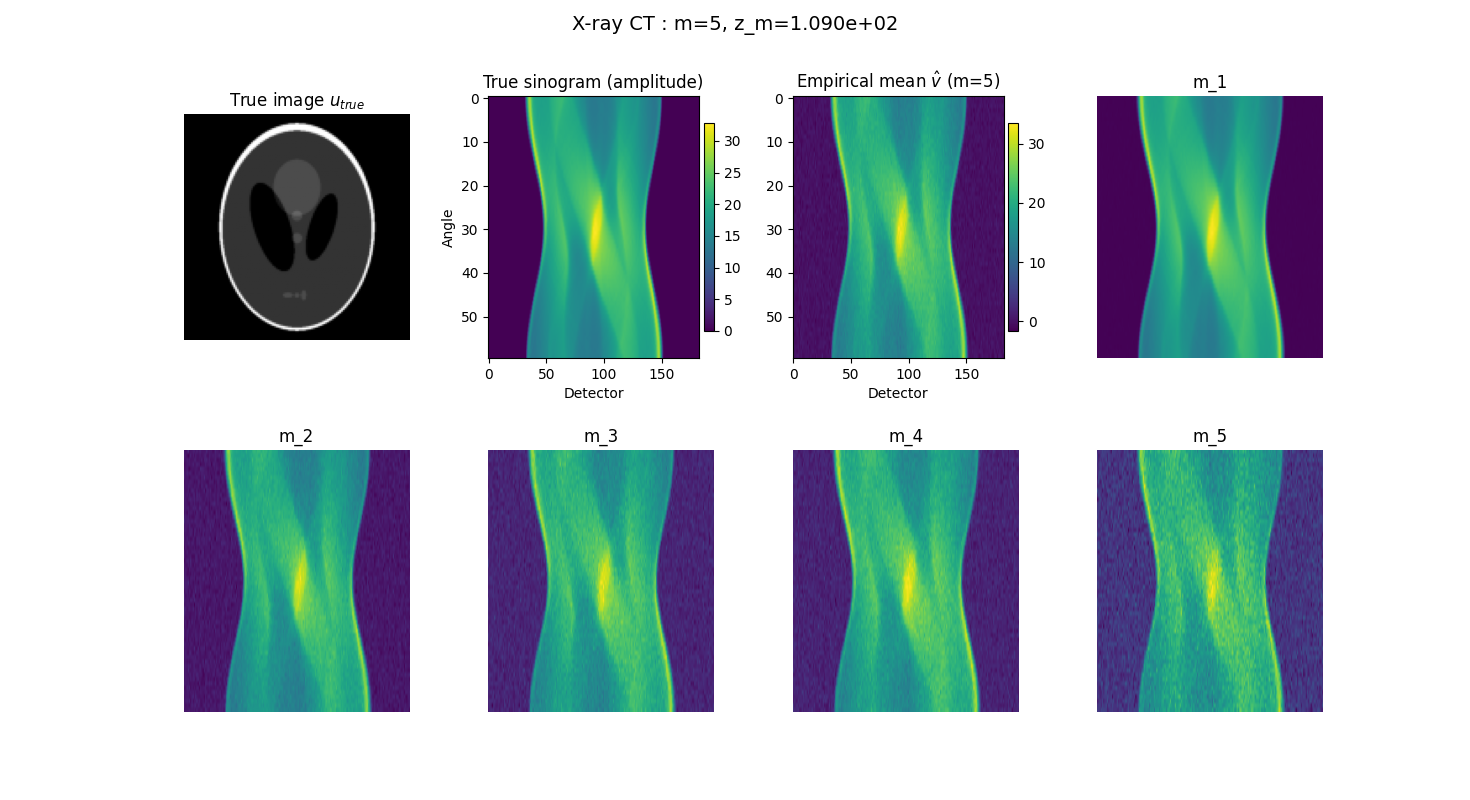}
    \caption{Exact solution for X-ray CT with 5 measured noisy sinograms and their empirical mean}
    \label{fig:CT_true}
\end{figure}
 Data acquisition simulates a parallel-beam CT geometry using \texttt{odl.tomo.parallel\_beam\_geometry} with $q_\phi = 60$ uniformly distributed projection angles over $[0, 2\pi)$.
The forward operator $\F_c$ is implemented via \texttt{odl.tomo.RayTransform}. For a $128 \times 128$ image ($n=16,384$ pixels), ODL sets $d = \lceil \sqrt{2} \cdot 128 \rceil = 182$ detector elements, resulting in $q = 60 \cdot 182 = 10,920$ measurements. This under determined system ($n > q$) constitutes an ill-conditioned linear system.
The parameters for~\eqref{eqn: edge weight fun} are empirically set to $R = 6$ and $\sigma = 0.05$. The step sizes $\alpha_k^\delta$ and $\beta_k^\delta$ are chosen according to \eqref{alpha} and \eqref{beta} with $\zeta_0 =0.2$, $\zeta_1 = 0.5$, $\nu_0 = \nu_1 = 0.05$, and $\tau_m = 2m^{0.5}$. For \cref{rule:heuristic_discrepancy}, we select $\varrho =100$.

\noindent First we generate the initial solution using FBP, Tik and TV. The initial solutions for the case $m=10$ are shown in Fig.~\ref{fig:initial}. Then we apply \texttt{E-IRMGL+$\Psi$}~\eqref{main iterative scheme} using both the stopping strategies and there reconstruction results are plotted in Fig.~\ref{fig:CT_m5_and_m10_rec} and Fig.~\ref{fig:CT_m50_and_m100_rec} for various cases of number of measurements $m$. 
\begin{figure}
    \centering
    \includegraphics[width=0.8\linewidth]{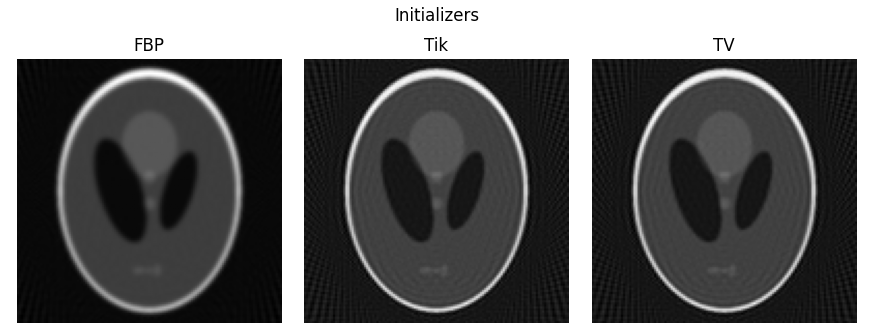}
    \caption{Reconstructed solutions of X-ray CT by using initializers with $m=10$}
    \label{fig:initial}
\end{figure}
\begin{figure}[htbp]
    \centering
    \begin{subfigure}[t]{0.48\textwidth}
        \centering
        \includegraphics[width=\textwidth]{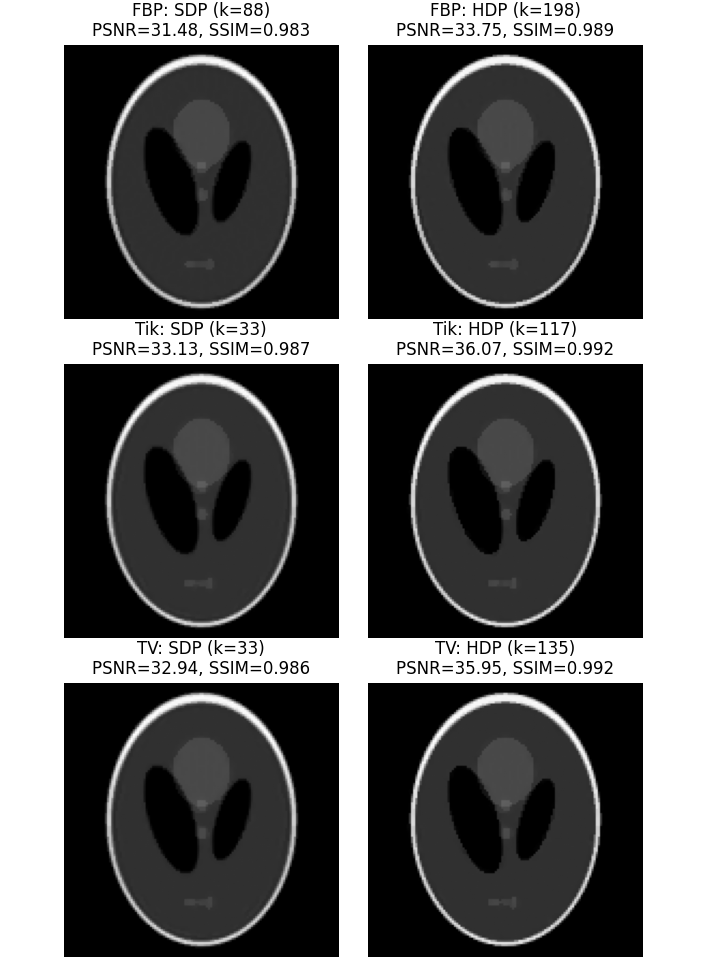}
        \caption{$m=5$}
    \end{subfigure}
    \hspace{-0.7cm}
    \begin{subfigure}[t]{0.471\textwidth}
        \centering
        \includegraphics[width=\textwidth]{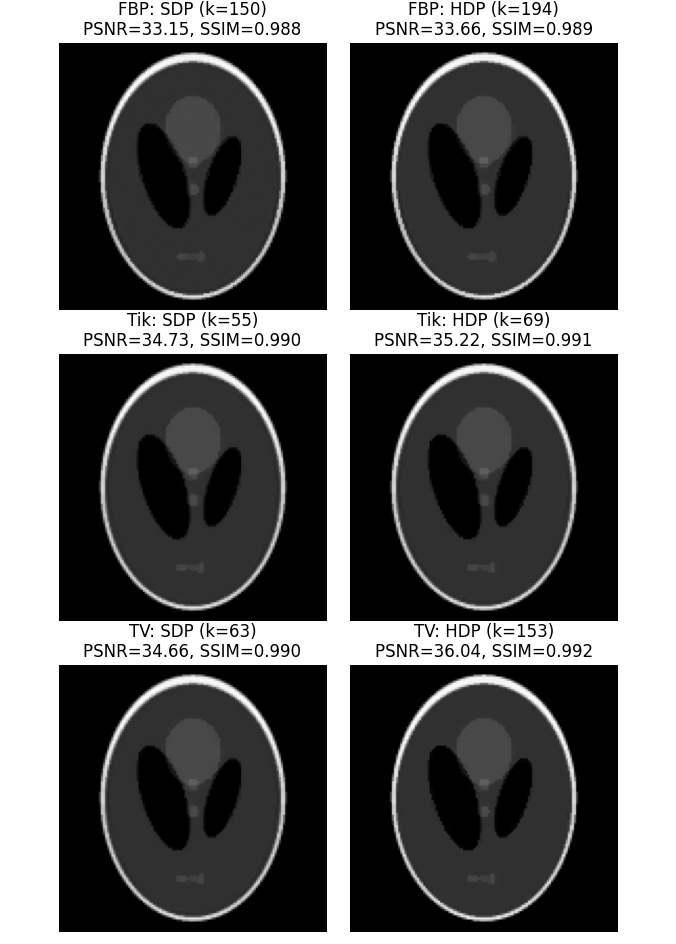}
        \caption{$m=10$}
    \end{subfigure}
    \caption{Reconstructed solutions of X-ray CT with $m=5$ and $m=10$}
 \label{fig:CT_m5_and_m10_rec}
  \end{figure}
\begin{figure}[htbp]
    \centering
    \begin{subfigure}[t]{0.458\textwidth}
        \centering
        \includegraphics[width=\textwidth]{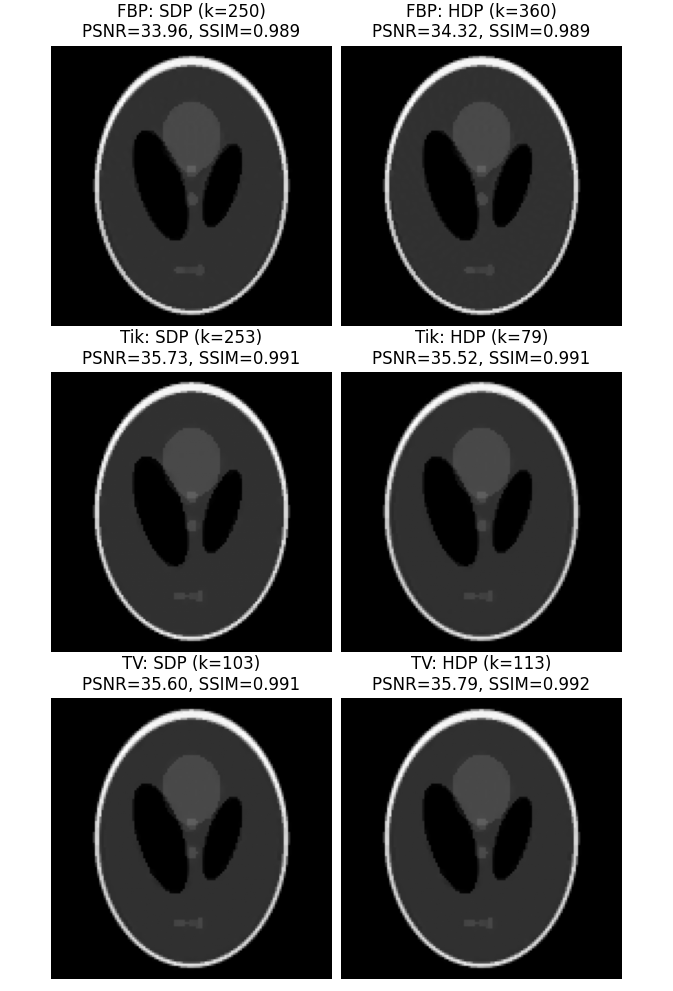}
        \caption{$m=50$}
    \end{subfigure}
    \hspace{-0.7cm}
    \begin{subfigure}[t]{0.48\textwidth}
        \centering
        \includegraphics[width=\textwidth]{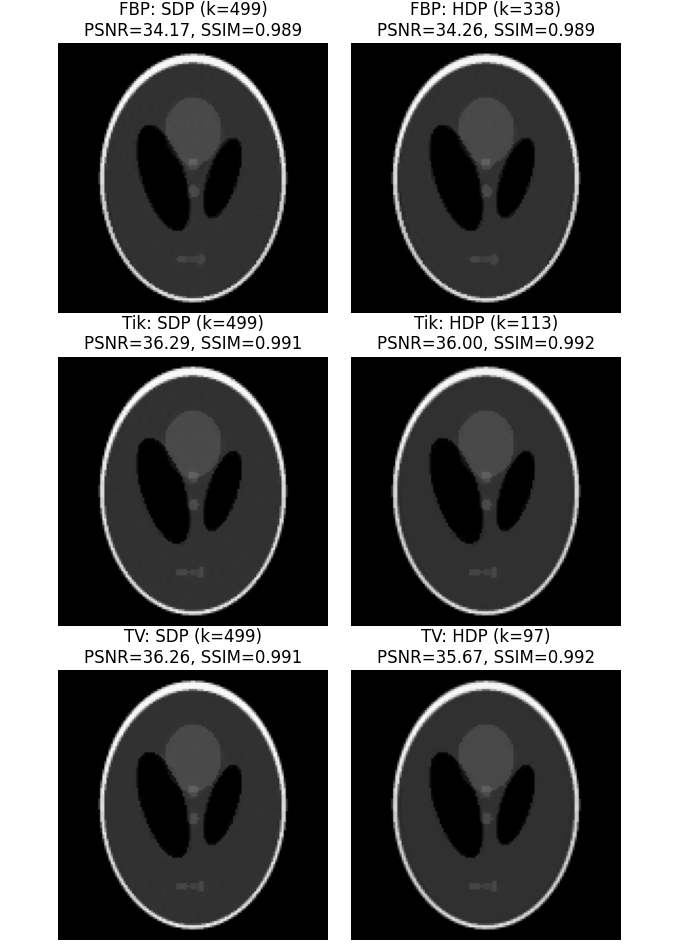}
        \caption{$m=100$}
    \end{subfigure}
        \caption{Reconstructed solutions of X-ray CT with $m=50$ and $m=100$}
  \label{fig:CT_m50_and_m100_rec}
  \end{figure}
 A detailed comparison of the reconstruction quality and the corresponding number of iterations is provided in \cref{tab:two_alg_comparison}. The results indicate that  \cref{rule:statistical_discrepancy} performs better for a smaller number of measurement points (e.g., \(m = 5\) and \(m = 10\)). However, as the number of measurements increases, the threshold value in \cref{rule:statistical_discrepancy} decreases, leading to a substantially larger number of iterations. In contrast, the stopping index determined by \cref{rule:heuristic_discrepancy} is reached more quickly, owing to the oscillatory nature of the convergence behavior after several iterations. For the cases \(m = 5\) and \(m = 10\), the maximum number of iterations was set to \(200\), while for \(m = 50\) and \(m = 100\) it was increased to \(500\). These choices were guided by empirical observations and aimed at balancing computational efficiency with memory requirements. 
\begin{table}[htb]
    \centering
    \caption{Numerical  comparison of both algorithms for X-ray CT by using different values of $m$.}
    \label{tab:two_alg_comparison}
    \renewcommand{\arraystretch}{1.15}
    \setlength{\tabcolsep}{6pt}
    \begin{tabular}{c l | r r r r | r r r r}
        \toprule
        \multirow{2}{*}{$\boldsymbol{m}$} & \multirow{2}{*}{\textbf{$\Psi$}} 
            & \multicolumn{4}{c|}{\textbf{\Cref{alg:buildtree}}} 
            & \multicolumn{4}{c}{\textbf{\Cref{alg:heuristic}}} \\
        & & \textbf{Iter.} & \textbf{RRE} & \textbf{PSNR} & \textbf{SSIM} 
          & \textbf{Iter.} & \textbf{RRE} & \textbf{PSNR} & \textbf{SSIM} \\
        \midrule

        \multirow{3}{*}{5}
            & \texttt{FBP}  & 88  &  0.114 & 31.48  & 0.9831  & 198 & 0.088 & 33.75  & 0.9889  \\
            & \texttt{Tik}  &33 & 0.099  & 33.13  & 0.9865  &   117 & 0.067 & 36.07  & 0.9916 \\
              & \texttt{TV}   & 33  & 0.097 & 32.94  & 0.9863  & 135 & 0.068  & 35.95  & 0.9917  \\
        \midrule

        \multirow{3}{*}{10}
            & \texttt{FBP}  & 150 &0.094 & 33.15  & 0.9879  & 194  & 0.089  & 33.66  & 0.9888  \\
            & \texttt{Tik}  & 55  & 0.078  & 34.73  & 0.9900  & 69  & 0.074   &  35.22& 0.9908  \\
            & \texttt{TV}   & 63  & 0.079  & 34.66  & 0.9901  & 153 & 0.068  & 36.04  & 0.9915 \\
        \midrule

        \multirow{3}{*}{50}
            & \texttt{FBP}  & 250  & 0.086  & 33.96  & 0.9891  & 360  & 34.32  & 0.9893  & 0.082  \\
            & \texttt{Tik}  & 253  & 0.071  & 35.73  & 0.9908  & 79  & 0.071  & 35.52  & 0.9912  \\
            & \texttt{TV}   & 103  & 0.071  & 35.60  & 0.9914  & 113  & 35.79  & 0.9916  & 0.069   \\
        \midrule

        \multirow{3}{*}{100}
            & \texttt{FBP}  & max  & 0.084 & 34.17  & 0.9888  & 338 & 0.083 & 34.26   & 0.9892  \\
            & \texttt{Tik}  & max  & 0.066 & 36.29  & 0.9914  & 113  & 0.068  & 36.00  & 0.9917  \\
            & \texttt{TV}   & max & 0.066  & 36.26  & 0.9911  & 97 & 0.070  & 35.67  & 0.9916   \\
        \bottomrule
    \end{tabular}
\end{table}
Other metrics like PSNR, SSIM and RRE are also plotted in Fig.~\ref{fig:m_5_met}, Fig.~\ref{fig:m_10}, Fig.~\ref{fig:m_50} and Fig.~\ref{fig:m_100} for each chosen value of $m.$ 
\begin{figure}[htbp]
    \centering
       \includegraphics[width=0.9\textwidth, height=0.4\textheight]{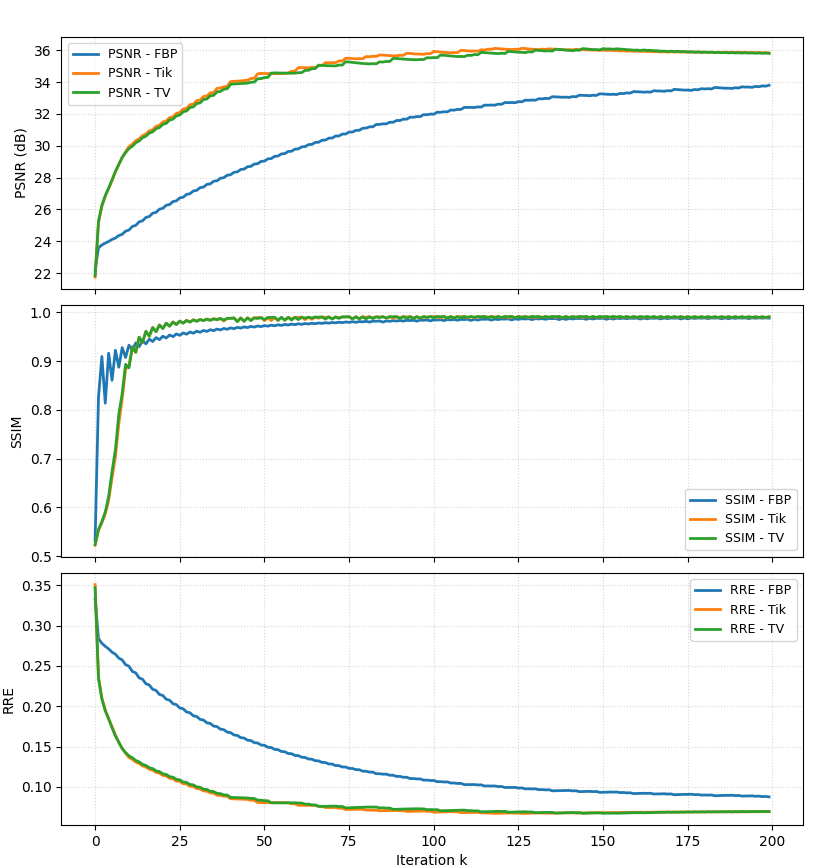} 
        \caption{PSNR, SSIM and RRE curves of X-ray CT for $m = 5$} 
        \label{fig:m_5_met}
    \end{figure}
\begin{figure}[htbp]
    \centering
        \includegraphics[width=0.9\textwidth, height=0.4\textheight]{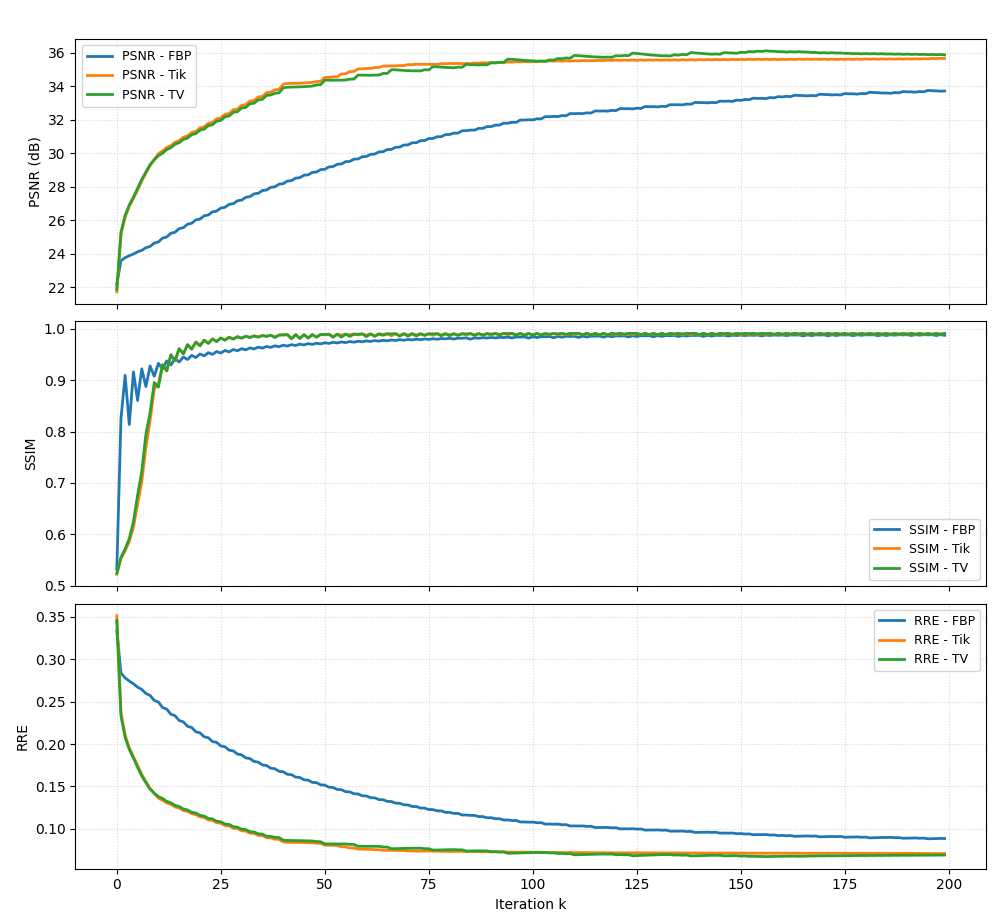} 
        \caption{PSNR, SSIM and RRE curves of X-ray CT for $m = 10$} 
        \label{fig:m_10}
    \end{figure}
\begin{figure}[htbp]
    \centering
      \includegraphics[width=0.9\textwidth, height=0.4\textheight]{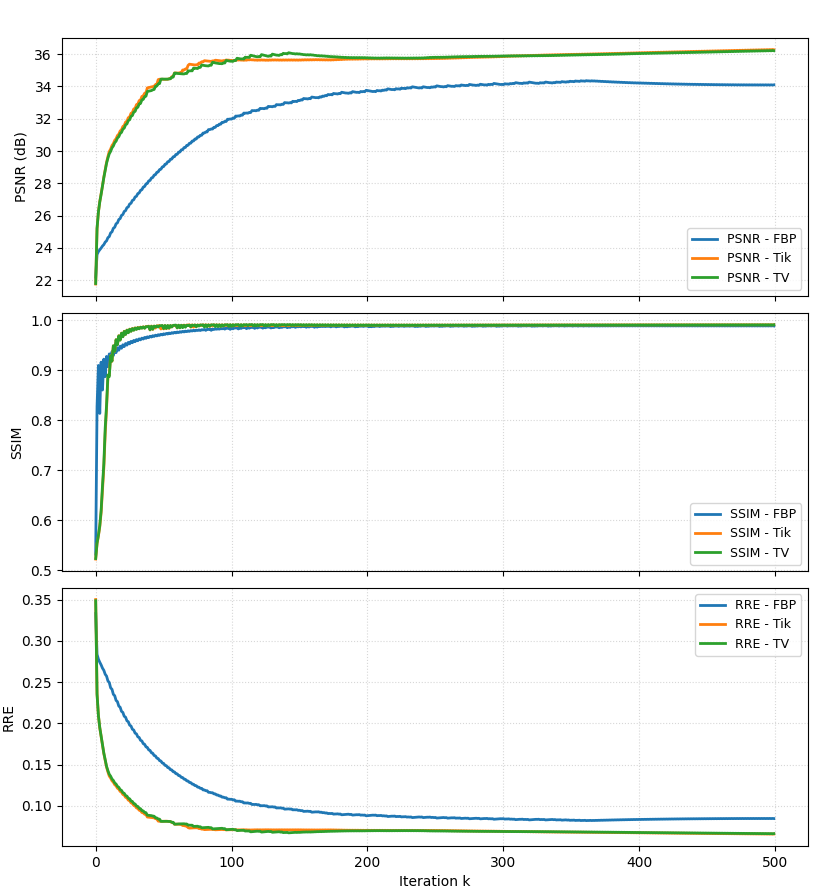} 
        \caption{PSNR, SSIM and RRE curves of X-ray CT for $m = 50$} 
        \label{fig:m_50}
    \end{figure}
\begin{figure}[htbp]
    \centering
            \includegraphics[width=0.9\textwidth, height=0.4\textheight]{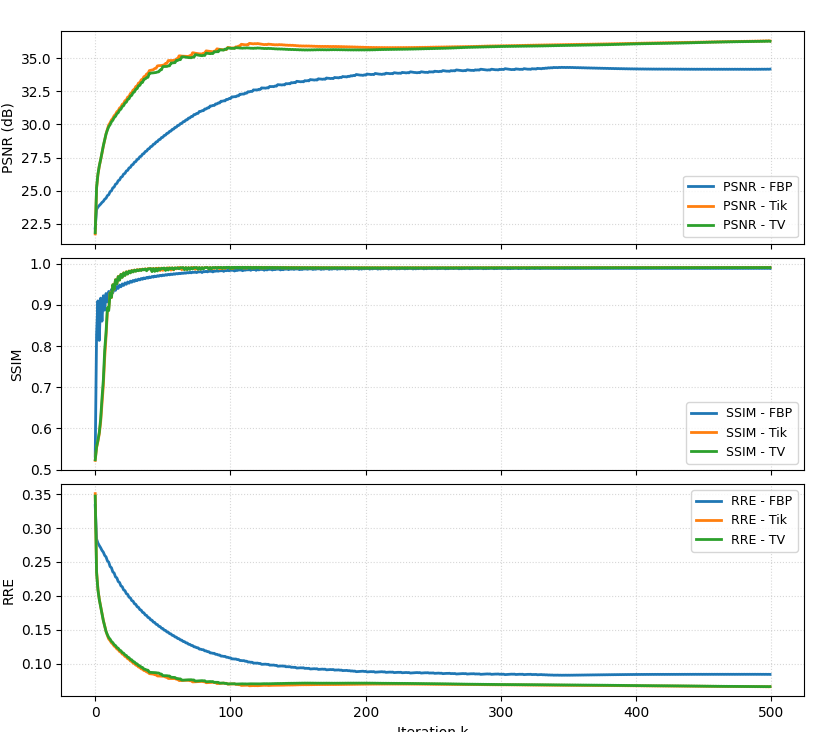} 
        \caption{PSNR, SSIM and RRE curves of X-ray CT for $m = 100$} 
        \label{fig:m_100}
    \end{figure}
\subsection{Phase retrieval CT} In this section, we consider the phase retrieval problem, which arises in numerous practical settings such as X-ray crystallography, microscopy, and related coherent imaging modalities, see \cite{fannjiang2020numerics} and the references therein. In conventional CT, detectors measure the attenuated X-ray intensity after logarithmic linearization, yielding a linear relation between the object and its projections, as described in \eqref{CT1}. This model presumes that both amplitude and phase information of the measured signal are accessible. In contrast, in coherent imaging techniques such as X-ray phase-contrast tomography, propagation-based imaging, and coherent diffraction tomography, the detector records only the intensity (magnitude squared) of the transmitted or scattered wave field. Consequently, the reconstruction task becomes inherently nonlinear and is formulated as a phase retrieval inverse problem.

\noindent
Mathematically, the phase retrieval problem in this context can be written as
\begin{equation}\label{eqn:model_PR}
\F_p(u)\coloneq |\mathcal{F}_c u|^2 = \tilde v,
\end{equation}
where $|\mathcal{F}_c u|^2$ denotes the element-wise squared magnitude of the projection data. Unlike the conventional linear CT model, this formulation is inherently nonlinear and nonconvex due to the loss of phase information. Consequently, the reconstruction of $u$ from intensity-only measurements is significantly more challenging, as multiple solutions may correspond to the same magnitude data.
To recover $u$, we employs \Cref{alg:buildtree} and \Cref{alg:heuristic}. The proposed schemes are designed to handle the ill-posedness and nonlinearity of the phase retrieval problem.

\noindent
Clearly $\mathcal F_p$ is continuous differentiable
and its Fréchet derivative 
$\F_p'(u): \mathbb{R}^n \to \mathbb{R}^q$
is the linear map
\[\F_p'(u)[h] = 2 \big( y \odot (\mathcal{F}_c h) \big) 
= 2 \big( \operatorname{diag}(y)\, \mathcal{F}_c \big) h,
\]
where $h \in \mathbb R^n, y = \mathcal{F}_c u$ and $\odot$ denotes the Hadamard (element wise) product.  
Thus, the Jacobian matrix is given by
\[
J(u) = \F_p'(u) = 2\, \operatorname{diag}(\mathcal{F}_c u)\, \mathcal{F}_c \in \mathbb{R}^{q \times n}.
\]
The adjoint of the derivative $\F_p'(u)$ is the map
$\F_p'(u)^*: \mathbb{R}^q \to \mathbb{R}^n$, 
given by
\[\F_p'(u)^*[w] 
= 2\, \mathcal{F}_c^{\top} \big( (\mathcal{F}_c u) \odot w \big),
\quad \text{for any } w \in \mathbb{R}^m.
\]
Equivalently, in matrix form,
\begin{equation}\label{eqn:adjoint_PR}
\F_p'(u)^* = 2\, \mathcal{F}_c^{\top} \operatorname{diag}(\mathcal{F}_c u).
\end{equation}
We use \eqref{eqn:model_PR} and \eqref{eqn:adjoint_PR} to apply \texttt{E-IRMGL+$\Psi$} method. The numerical simulations are conducted using the \texttt{binary\_blobs} available in the \texttt{skimage.data} library. The ground truth image and the corresponding noisy sinograms, contaminated with Gaussian noise \(\epsilon_i \sim \mathcal{N}(0, \varepsilon^2)\) with \(\varepsilon = 10\), are illustrated in Fig.~\ref{fig:PR_true}.
 We set \(\tau_m = 3m^{0.5}\) for \cref{rule:statistical_discrepancy}, while for \cref{rule:heuristic_discrepancy}, the parameter \(\varrho\) is chosen as \(1000\) and the maximum number of iteration is chosen to be $300$.
\begin{figure}
    \centering
    \includegraphics[width=0.98\linewidth]{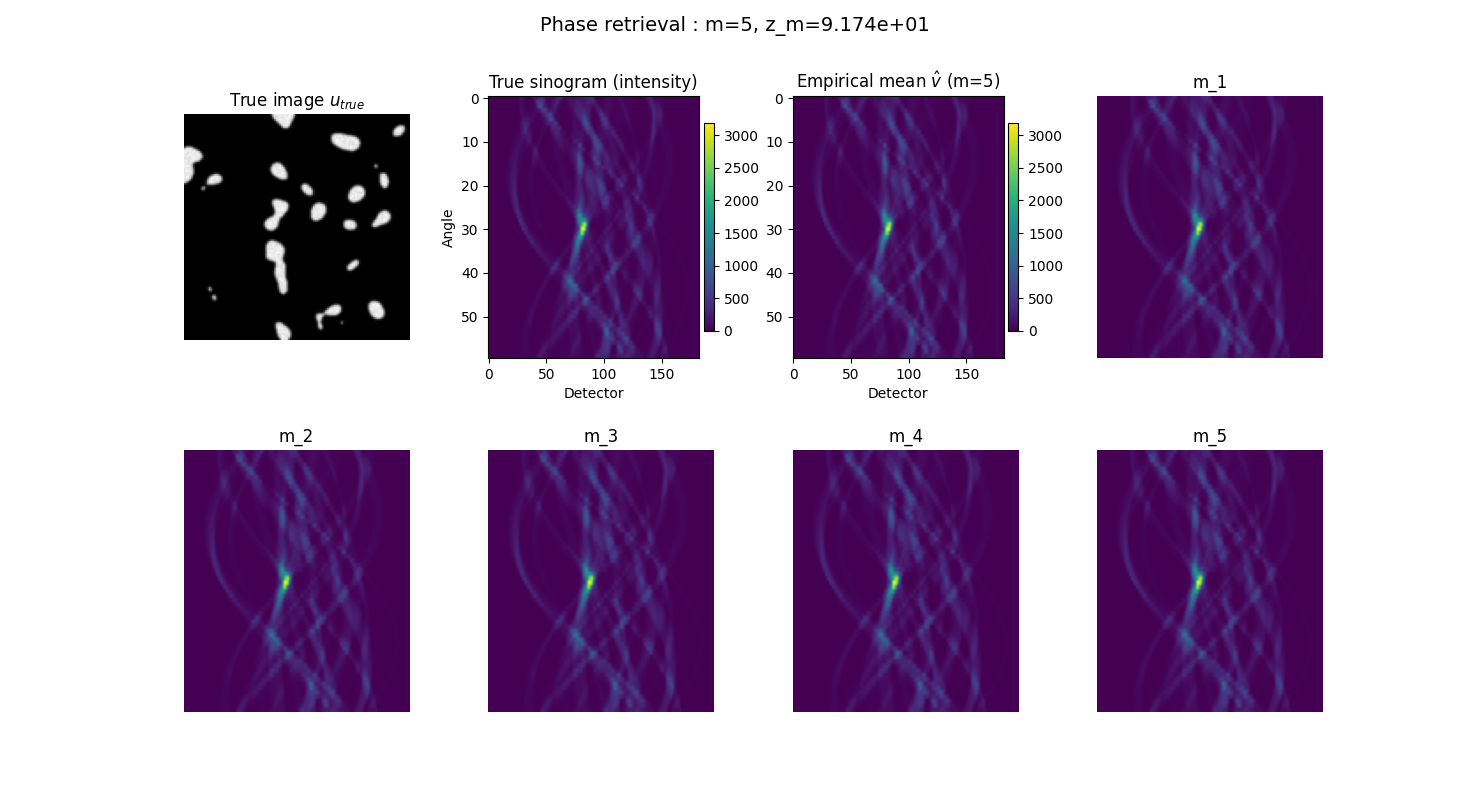}
    \caption{Exact solution for phase retrieval CT with 5 measured sinograms and their empirical mean}
    \label{fig:PR_true}
\end{figure}
\Cref{alg:buildtree} and \Cref{alg:heuristic} are  initialized using FBP, Tikhonov, and TV methods, with the results for \(m = 10\) displayed in Fig.~\ref{fig:PR_m_10_initilaizers}. Subsequently, the \texttt{E-IRMGL+$\Psi$} scheme \eqref{main iterative scheme} is applied with both stopping strategies, and the corresponding reconstructions for different measurement numbers \(m\) are shown in Fig.~\ref{fig:PR_rec}. In \cref{tab:two_alg_comparison_PR}, we summarizes the performance comparison of the two stopping strategies in the context of the phase retrieval problem.

\begin{figure}
    \centering
    \includegraphics[width=0.8\linewidth]{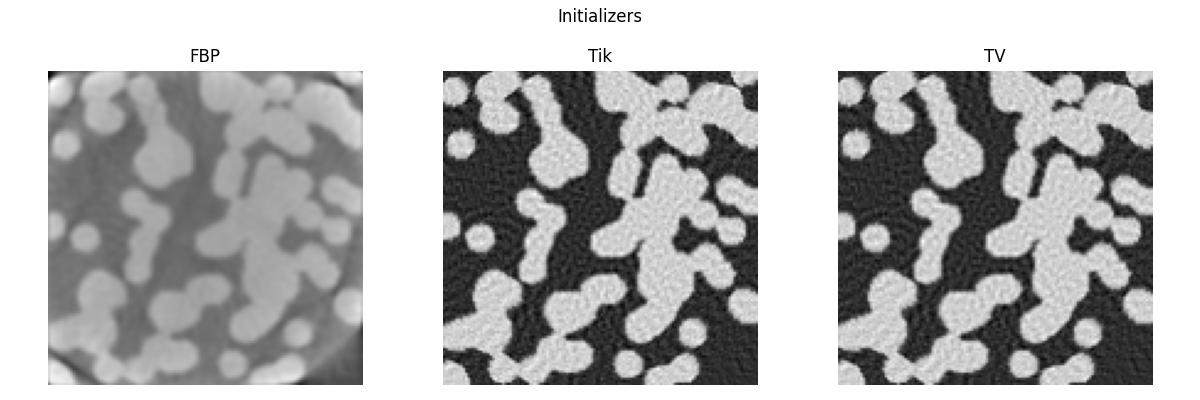}
    \caption{Reconstructed solutions of phase retrieval CT by using initializers with $m=10$}
    \label{fig:PR_m_10_initilaizers}
\end{figure}
\begin{table}[htb]
    \centering
    \caption{Numerical  comparison of both algorithms for phase retrieval CT by using different values of $m$.}
    \label{tab:two_alg_comparison_PR}
    \renewcommand{\arraystretch}{1.15}
    \setlength{\tabcolsep}{6pt}
    \begin{tabular}{c l | r r r r | r r r r}
        \toprule
        \multirow{2}{*}{$\boldsymbol{m}$} & \multirow{2}{*}{\textbf{$\Psi$}} 
            & \multicolumn{4}{c|}{\textbf{\Cref{alg:buildtree}}} 
            & \multicolumn{4}{c}{\textbf{\Cref{alg:heuristic}}} \\
        & & \textbf{Iter.} & \textbf{RRE} & \textbf{PSNR} & \textbf{SSIM} 
          & \textbf{Iter.} & \textbf{RRE} & \textbf{PSNR} & \textbf{SSIM} \\
        \midrule
        \multirow{3}{*}{5}
            & \texttt{FBP}  & 150  &  0.155 & 19.21  & 0.7669  & 292 & 0.235 & 20.40  & 0.7947  \\
            & \texttt{Tik}  &98 & 0.145  & 19.79  & 0.7636  &   274 & 0.124 & 21.08  & 0.8008 \\
              & \texttt{TV}   & 97  & 0.147 & 19.67  & 0.7605  & 277 & 0.125  & 21.07  & 0.8017  \\
        \midrule
        \multirow{3}{*}{10}
            & \texttt{FBP}  & 198 & 0.143 & 19.88  & 0.7786  & 286  & 0.133  & 20.54  & 0.7900  \\
            & \texttt{Tik}  & 152  & 0.135  & 20.43  & 0.7777  & 280  & 0.122   &  21.23& 0.7947  \\
            & \texttt{TV}   & 134  & 0.136  & 20.43  & 0.7771  & 295 & 0.122  & 21.27  & 0.7970 \\
        \bottomrule
    \end{tabular}
\end{table}
\begin{figure}[htbp]
    \centering
    \begin{subfigure}[t]{0.476\textwidth}
        \centering
        \includegraphics[width=\textwidth]{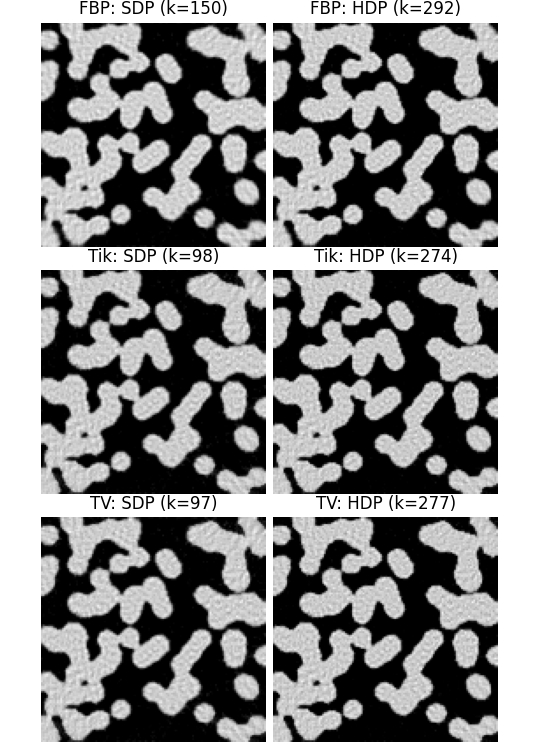}
        \caption{$m=5$}
    \end{subfigure}
    \hspace{-0.7cm}
    \begin{subfigure}[t]{0.48\textwidth}
        \centering
        \includegraphics[width=\textwidth]{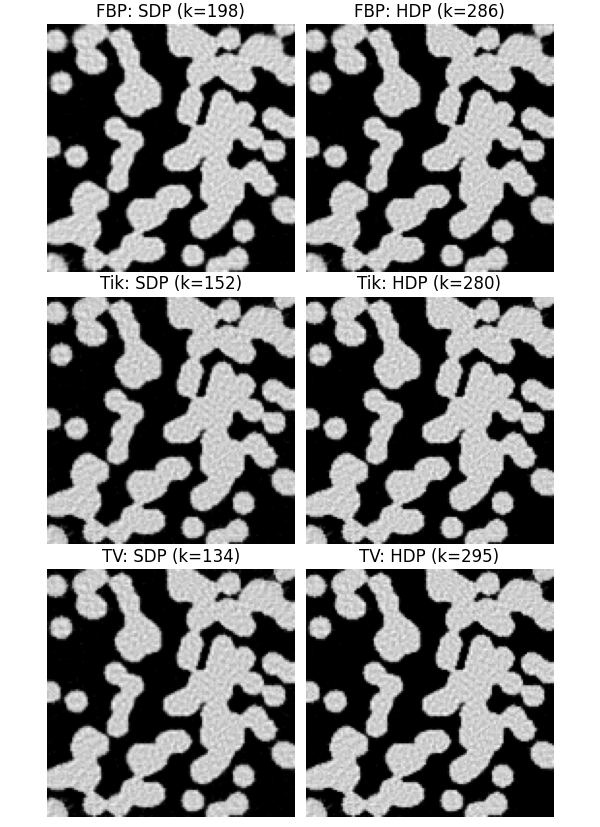}
        \caption{$m=10$}
    \end{subfigure}
        \caption{Reconstructed solutions of phase retrieval CT with $m=50$ and $m=100$}
        \label{fig:PR_rec}
  \end{figure}

\section{Concluding remarks}
\label{sec:conclusions}
\noindent This study introduces and analyzes a novel graph-based iterative regularization method for solving nonlinear ill-posed inverse problems. 
The proposed method is examined under two distinct stopping criteria, namely the statistical discrepancy principle and the heuristic discrepancy principle. 
Unlike conventional approaches, the framework does not require the knowledge of the noise level. Instead, it employs the mean of i.i.d.\ estimators to approximate it. 
Comprehensive analyses of stability, regularity, and convergence are carried out under both stopping rules, based on standard local assumptions on the forward operator~$\F$ and general conditions on initializers~$\Psi_\theta$. 
The theoretical findings are further supported and validated through detailed numerical experiments and comparative studies on two numerical examples, including X-ray CT and phase retrieval CT problems.

The proposed framework has been developed within a finite-dimensional setting. As a potential direction for future research, it would be of significant interest to extend the convergence analysis to infinite-dimensional Hilbert or Banach space frameworks.

\section*{Data availability}
On reasonable request, the authors will provide the datasets and source code created during this study.

\bibliographystyle{siamplain}
\bibliography{references}
\appendix
\section{}\label{sec:appendix}
\noindent In this appendix we collect some useful results that are required to proof Theorem~\ref{thm: non noisy}.
  \begin{lemma}[Calibration of the budget]\label{lem:calibration}
For $j,m\in\mathbb N$ set $a_{j,m}:=\mathbb E\big[\|u^{(m)}_j-u_j\|^{2}\big]$. Then, under the assumptions of Lemma~\ref{lemma: stability}, there exists a nondecreasing
$K_{\max}:\mathbb N\to\mathbb N$ with $K_{\max}(m)\to\infty$ and
\begin{equation}\label{eq:budget}
  \lim_{m\to\infty}\;\big(K_{\max}(m)+1\big)\max_{0\le j\le K_{\max}(m)}a_{j,m}\;=\;0.
\end{equation}
\end{lemma}
\begin{proof}
By Lemma~\ref{lemma: stability}, we have \(\lim_{m\to\infty} a_{j,m}=0 \) for every fixed \(j\geq 0\).
Now fix \(n\geq 1\). Since
\(
\max_{0\leq j\leq n} a_{j,m}
\)
is the maximum of finitely many sequences converging to zero, it follows that
\[
\lim_{m\to\infty}\max_{0\leq j\leq n} a_{j,m}=0.
\]
Hence, for each \(n\), there exists an integer \(m_n\) such that
\[
\max_{0\leq j\leq n} a_{j,m}
    \leq \frac{1}{(n+1)^2},
    \qquad \forall\, m\geq m_n.
\]
Without loss of generality, by replacing \(m_n\) with
\(
\max\{m_1,\ldots,m_n\}+n,
\)
we may assume that the sequence \((m_n)_{n\geq1}\) is strictly increasing.
Define
\[
K_{\max}(m):=
\begin{cases}
\max\{n\geq1:\,m_n\leq m\}, & m\geq m_1,\\
0, & m<m_1.
\end{cases}
\]
By construction, \(K_{\max}\) is nondecreasing and satisfies
\(
K_{\max}(m)\to \infty \) as \(m\to\infty.
\)
Moreover, if \(m_n\leq m<m_{n+1}\), then \(K_{\max}(m)=n\), and therefore
\[
\bigl(K_{\max}(m)+1\bigr)
\max_{0\leq j\leq K_{\max}(m)}a_{j,m}
=(n+1)\max_{0\leq j\leq n}a_{j,m}
\leq
(n+1)\frac{1}{(n+1)^2}
=
\frac{1}{n+1}.
\]
Since \(n\to\infty\) whenever \(m\to\infty\), we conclude that
\[
\bigl(K_{\max}(m)+1\bigr)
\max_{0\leq j\leq K_{\max}(m)}a_{j,m}
\to 0,
\]
which proves the claim.
\end{proof}
\begin{proposition}[Uniform a priori bound]\label{prop:apriori}
Given the budget of Lemma~\ref{lem:calibration} there is a constant
$M:=1+\|u_0-u^\dagger\|$, independent of $m$, such that for all sufficiently
large $m$
\[
  \mathbb E\Big[\max_{0\le j\le K_{\max}(m)}\big\|u^{(m)}_j-u^\dagger\big\|^{2}\Big]^{1/2}\;\le\;M,
\]
which immediately implies that \(
  \mathbb E\Big[\big\|u^{(m)}_{k_m}-u^\dagger\big\|^{2}\Big]\le M^{2} .
\)
\end{proposition}

\begin{proof}
Write $K:=K_{\max}(m)$. Since a maximum over $j\le K$ is dominated by the
corresponding sum,
\[
  \mathbb E\Big[\max_{0\le j\le K}\big\|u^{(m)}_j-u_j\big\|^{2}\Big]
  \;\le\;\sum_{j=0}^{K}\mathbb E\Big[\big\|u^{(m)}_j-u_j\big\|^{2}\Big]
  \;\le\;(K+1)\max_{0\le j\le K}a_{j,m}\;\to \;0
\]
by \eqref{eq:budget}. In particular the left-hand side is $\le1$ for $m$ large.
Combining with Lemma~\ref{lemma: monotonicity for exact data} through Minkowski's inequality,
\[
  \mathbb E\Big[\max_{0\le j\le K}\big\|u^{(m)}_j-u^\dagger\big\|^{2}\Big]^{1/2}
  \le\mathbb E\Big[\max_{0\le j\le K}\big\|u^{(m)}_j-u_j\big\|^{2}\Big]^{1/2}
   +\|u_0-u^\dagger\|\le M .
\]
The final assertion follows from $k_m\le K_{\max}(m)$ pointwise, which gives
$$\|u^{(m)}_{k_m}-u^\dagger\|\le\max_{j\le K_{\max}(m)}\|u^{(m)}_j-u^\dagger\|\qquad \text{almost surely.}$$
\end{proof}
\begin{lemma}\label{lem:noearlystop}
Let $j\in\mathbb N$ be fixed with $\F(u_j)\neq v$. Then under the assumptions of Lemma~\ref{lemma: stability}
\[\displaystyle\lim_{m\to\infty}\mathbb P\big(k_m=j\big)=0.\]
\end{lemma}
\begin{proof}
Set $c_j:=\|\F(u_j)-v\|>0$, which is deterministic because $j$ is fixed.
By the stopping rule \eqref{eq:stopping_rule} and the triangle inequality,
\[
  \{k_m=j\}\subseteq
  \Big\{\big\|\F(u^{(m)}_j)-\hat v^{(m)}\big\|\le\tfrac{\tau_m}{\sqrt m}z_m\Big\}
  \subseteq\big\{S_m\ge c_j\big\},
\]
where  \(S_m:=\big\|\F(u_j)-\F(u^{(m)}_j)\big\|+\tfrac{\tau_m}{\sqrt m}z_m +\big\|\hat v^{(m)}-v\big\|\),
so it suffices to prove $S_m\xrightarrow{\mathbb P}0$. The first summand does
so because Lemma~\ref{lemma: stability} gives $u^{(m)}_j\xrightarrow{\mathbb P}u_j$, while $\F$ is continuous at
$u_j$ by Assumption~\ref{assump:Fi}. For the other two, Jensen's inequality
gives $\mathbb E[z_m]\le\mathbb E[z_m^{2}]^{1/2}=\sigma$ and
$\mathbb E[\|\hat v^{(m)}-v\|]\le\sigma/\sqrt m$, whence, by Markov's
inequality and $\tau_m/\sqrt m\to0$,
\[
  \mathbb P\Big(\tfrac{\tau_m}{\sqrt m}z_m\ge\varepsilon\Big)
  \le\frac{\tau_m\,\sigma}{\varepsilon\sqrt m}\to 0,
  \qquad
  \mathbb P\big(\|\hat v^{(m)}-v\|\ge\varepsilon\big)
  \le\frac{\sigma}{\varepsilon\sqrt m}\to0
  \qquad(\varepsilon>0).
\]
Hence $\mathbb P(k_m=j)\le\mathbb P(S_m\ge c_j)\to0$.
\end{proof}
\end{document}